\documentclass[11 pt,final]{article}

\usepackage{amsmath,amssymb,amsthm,amscd,amsfonts}
\usepackage{epsfig,pinlabel}
\usepackage[hmargin=1.25in, vmargin=1in]{geometry}
\usepackage[font=small,format=plain,labelfont=bf,up,textfont=it,up]{caption}
\usepackage{subcaption}
\usepackage{stmaryrd}
\usepackage{verbatim}
\usepackage{type1cm}
\usepackage{url}
\usepackage{calc}
\usepackage[all,cmtip]{xy}
\usepackage{graphicx}
\usepackage{multicol}
\usepackage{pstricks}
\usepackage{relsize}
\usepackage{float}
\usepackage{overpic}
\usepackage{enumerate}
\usepackage[affil-it]{authblk}
\usepackage{mathtools}

\newcommand{\urlwofont}[1]{\urlstyle{same}\url{#1}}

\newcommand{\nc}{\newcommand}
\nc{\nt}{\newtheorem}
\nc{\dmo}{\DeclareMathOperator}

\theoremstyle{plain}
\newtheorem{theorem}{Theorem}[section]
\newtheorem{maintheorem}{Theorem}
\newtheorem{proposition}[theorem]{Proposition}
\newtheorem{lemma}[theorem]{Lemma}

\newtheorem{corollary}[theorem]{Corollary}

\newtheorem{question}[theorem]{Question}
\newtheorem{problem}[theorem]{Problem}
\newtheorem{claim}{Claim}

\theoremstyle{definition}

\theoremstyle{remark}
\newtheorem{remark}[theorem]{Remark}

\newtheorem{example}[theorem]{Example}

\DeclareMathOperator{\Mod}{Mod}

\dmo{\SMod}{SMod}
\dmo{\PMod}{PMod}
\dmo{\SHomeo}{SHomeo}
\dmo{\SI}{\mathcal{SI}}
\dmo{\SSp}{SSp}
\dmo{\PSp}{PSp}

\newcommand\R{\ensuremath{\mathbb{R}}}
\newcommand\C{\ensuremath{\mathbb{C}}}
\newcommand\Z{\ensuremath{\mathbb{Z}}}

\nc{\p}[1]{\noindent {\bf #1.}}
\nc{\margin}[1]{\marginpar{\scriptsize #1}}

\nc{\PartialIBases}{\mathfrak{IB}}
\nc{\PartialIBasesEx}{\widehat{\mathfrak{IB}}}
\nc{\PartialBases}{\mathfrak{B}}
\nc{\Building}{\mathfrak{T}}
\nc{\height}{\ensuremath{\text{ht}}}
\nc{\Poset}{\mathfrak{P}}
\nc{\Field}{\mathbb{F}}
\nc{\Link}{\ensuremath{\text{Link}}}
\nc{\Star}{\ensuremath{\text{Star}}}

\nc{\SymTorelli}{\ensuremath{\mathcal{SI}}}
\nc{\BTorelli}{\ensuremath{\mathcal{BI}}}
\dmo{\Braid}{\ensuremath{B}}
\dmo{\PureBraid}{\ensuremath{PB}}
\nc{\Hyper}{\ensuremath{\iota}}

\nc{\BigFreeProd}{\mathop{\mbox{\huge{$\ast$}}}}

\nc{\Quotient}{\ensuremath{\mathcal{Q}}}
\nc{\QuotientEx}{\ensuremath{\widehat{\mathcal{Q}}}}

\nc{\Presentation}[2]{\ensuremath{\text{$\langle #1$ $|$ $#2 \rangle$}}}
\nc{\SpGen}{\ensuremath{S_{\text{Sp}}}}
\nc{\SpRel}{\ensuremath{R_{\text{Sp}}}}
\nc{\QGen}{\ensuremath{S_{\mathcal{Q}}}}
\nc{\QRel}{\ensuremath{R_{\mathcal{Q}}}}
\nc{\PBs}{\ensuremath{T}}
\nc{\Qs}{\ensuremath{\overline{s}}}

\dmo{\PB}{PB}
\nc{\BIredg}{\mathcal{BI}_{2g+1}^{\text{red}}}
\nc{\BI}{\mathcal{BI}}
\dmo{\D}{D}
\dmo{\Stab}{Stab}
\dmo{\Surger}{Surger}
\nc{\I}{\mathcal{I}}
\renewcommand{\C}{\mathcal{C}}

\nc{\spanmap}{span}

\nc{\genbygen}[2]{\premonoid{#1}{#2}}
\nc{\premonoid}[2]{#1 \circledcirc #2}
\nc{\monoid}[2]{#1 \odot #2}

\nc{\G}{\Gamma}
\nc{\raag}{A_\Gamma}
\nc{\racg}{W_\G}
\nc{\raagdelt}{A_\Delta}
\dmo{\Aut}{Aut}
\dmo{\Out}{Out}
\nc{\Autraag}{\Aut(\raag)}
\nc{\Outraag}{\Out(\raag)}
\nc{\diag}{D_\G}
\nc{\Autraagdelt}{\Aut(A_\Delta)}
\nc{\glk}{\GL(k,\mathbb{Z})}
\nc{\GLn}{\GL_n(\mathbb{Z})}
\nc{\GLnt}{\GL_n(\mathbb{Z} / 2)}
\nc{\glkdelt}{\GL(k |\Delta| ,\mathbb{Z})}
\nc{\gldelt}{\GL(|\Delta| ,\mathbb{Z})}
\nc{\zkdelt}{\mathbb{Z}^{k|\Delta|}}
\nc{\join}{\mathcal{J}}
\nc{\pc}{\mathrm{PC}}
\dmo{\lk}{lk}
\dmo{\st}{st}
\dmo{\Inn}{Inn}
\newcommand{\what}{\widehat}

\nc{\pia}{\Pi \mathrm{A}}
\nc{\piaG}{\pia_\G}
\nc{\ppia}{\mathrm{P \Pi A}}
\nc{\ppiaG}{\mathrm{P \Pi A}_\G}
\nc{\epia}{\mathrm{E \Pi A}}
\nc{\epiaG}{\mathrm{E \Pi A}_\G}
\nc{\ptor}{\mathcal{PI}}
\nc{\ptorG}{\mathcal{PI}_\G}
\nc{\CGi}{C_\G(\iota)}
\nc{\pbc}{\mathfrak{B}^\pi}
\dmo{\Cay}{Cay}
\dmo{\rev}{rev}
\nc{\Autfn}{\Aut(F_n)}
\nc{\Outfn}{\Out(F_n)}
\dmo{\supp}{supp}
\dmo{\rk}{\mathrm{rk}}
\dmo{\PCT}{\mathrm{PCT}(\raag)}
\dmo{\PCTo}{\overline{\mathrm{PCT}}(\raag)}

\nc{\HOut}{\mathrm{HOut}(F_n)}
\nc{\STn}{\mathcal{ST}(n)}
\nc{\HMn}{\mathrm{HM}_n(\Z)}
\nc{\loc}{\mathcal{L}_n}
\nc{\loci}{\loc ^{\iota}}
\nc{\hloc}{\mathcal{HL}_n}
\nc{\hloci}{\hloc ^{\iota}}

\newcommand{\Sa}{\mathbb{S}}
\newcommand{\calI}{\mathcal{I}}
\newcommand{\calQ}{\mathcal{Q}}
\newcommand{\MG}{\mathcal{G}}

\newcommand{\Isom}{\text{Isom}}
\newcommand{\calT}{\mathcal{T}}
\newcommand{\GL}{\text{GL}}
\newcommand{\SL}{\text{SL}}
\newcommand{\MT}{\mathcal{A}}
\newcommand{\SO}{\text{SO}}

\newcommand{\overl}[1]{\overline{#1}}

\title{\vspace{-30pt} Hyperelliptic graphs and the period mapping on outer space}

\author{Corey Bregman and Neil J. Fullarton}

\begin{document}


%

\date{\today}

\newcounter{enumi_saved}

\maketitle
\vspace{-30pt}

\begin{abstract}The period mapping assigns to each rank $n$, marked metric graph $\G$ a positive definite quadratic form on $H_1(\G, \R)$. This defines maps $\hat \Phi$ and $\Phi$ on Culler--Vogtmann's outer space $CV_n$, and its Torelli space quotient $\mathcal{T}_n$, respectively. The map $\Phi$ is a free group analog of the classical period mapping that sends a marked Riemann surface to its Jacobian. In this paper, we analyze the fibers of $\Phi$ in $\mathcal{T}_n$, showing that they are aspherical, $\pi_1$-injective subspaces. Metric graphs admitting a \emph{hyperelliptic involution} play an important role in the structure of $\Phi$, leading us to define the \emph{hyperelliptic Torelli group}, $\STn \leq \Outfn$. We obtain generators for $\STn$, and apply them to show that the connected components of the locus of `hyperelliptic' graphs in $\mathcal{T}_n$ become simply-connected when certain degenerate graphs at infinity are added. \end{abstract}


\section{Introduction}

The focus of this paper is the period mapping $\hat \Phi$, which assigns to each marked, rank-$n$ metric graph $\G$ an inner product on the graph's first homology, $H_1(\G, \R)$. This maps Culler--Vogtmann's `outer space' of such graphs to the symmetric space of positive definite real quadratic forms. The map $\hat \Phi$ descends to a map $\Phi$ on the Torelli space quotient $\mathcal{T}_n$ of $CV_n$, which consists of \emph{homology}-marked graphs (cf. \cite{BBM07}). It is with the map $\Phi$ on $\mathcal{T}_n$ that we will be more interested, and so we abuse terminology, also referring to $\Phi$ as the period mapping. 

The map $\Phi$ arises from an inner product on the homology of a graph, which was introduced by Kotani--Sunada \cite{KoSu00}, Nagnibeda \cite{Na97} and Caporaso--Viviani \cite{CaVi11} and was previously known to Culler--Vogtmann. It is the free group analog of the classical period mapping $J$ of algebraic geometry, which assigns to a homology-marked Riemann surface its period matrix in the Siegel upper half-plane. It is difficult in general to describe the image of the classical map $J$; this is the so-called Schottky problem, which has only been completely resolved in genus $\leq 4$ \cite{Gr12}. In contrast, the fibers of $J$ are well-understood, as it is a 2-to-1 branched cover onto its image. These problems' relative difficulties are inverted for the map $\Phi$: its image is in a sense completely classified, using graphical matroids (cf. \cite{Vall03}), but its fibers are often more complicated than mere discrete sets. Our first theorem characterizes the topological structure of these fibers.

\begin{maintheorem}\label{fibers}The connected components of any fiber of the period mapping $\Phi$ are aspherical, $\pi_1$-injective subspaces of Torelli space, $\mathcal{T}_n$.\end{maintheorem}

The theorem is proved by showing that each fiber's connected components admit the structure of a quasi-fibration over a product of simplices, where the homotopy-fibers are configuration spaces of wedges of graphs. We prove this by decomposing a marked graph $\G$ in $\mathcal{T}_n$ along separating vertices and separating pairs of edges. It is then established that where and how these vertices and edges appear in $\G$ fully determines the fiber's connected component containing $\G$. 

Previously, Caporaso--Viviani gave a description of the fibers of a related mapping $\overl{ \Phi}$ on the full quotient $CV_n /\Outfn$ in terms of Whitney moves and 3-edge connectivizations \cite{CaVi11}.  In contrast, Theorem~\ref{fibers} provides a geometric description of the fibers on $\mathcal{T}_n$, and extends previous work of Owen Baker who, in his thesis \cite{Bak11}, explicitly calculated the possible fibers of $\Phi$ for $n=3$. 

Since 1-dimensional tropical varieties are finite metric graphs, $\Phi$ can also be seen as a tropical period map (see \cite{Mi06}, \cite{MiZh08}  for the tropical viewpoint).  Unlike the classical map $J$, the period mapping $\Phi$ is not a branched cover onto its image; indeed, the proof of Theorem~\ref{fibers} yields many examples of non-discrete fibers. However, when restricted to a `large' subspace of $\mathcal{T}_n$, the map $\Phi$ becomes a 2-to-1 branched cover onto its image. The branching occurs along the locus of so-called `hyperelliptic graphs', which consists of (homology-marked) graphs admitting an order 2 isometry that acts as negation on the graph's first homology. These graphs, and the map $\overl{\Phi}$, were studied previously by Chan \cite{Cha12}, \cite{Cha13}.

Hyperelliptic graphs are a geometric incarnation of a certain group $\HOut$ of `hyperelliptic' outer automorphisms of the free group $F_n$. Fix $X$ as a preferred free basis of $F_n$, and let $\iota$ denote the `hyperelliptic' involution in $\Autfn$ that inverts each $x \in X$. An automorphism $[ \theta ] \in \Outfn$ is \emph{hyperelliptic} if it centralizes the involution $[\iota] \in \Outfn$, and $\theta (X)$ is called a \emph{hyperelliptic basis} of $F_n$. We let $\STn$ denote the subgroup of hyperelliptic automorphisms inside the Torelli group $\mathcal{I}_n \leq \Outfn$. The groups $\HOut$ and $\STn$ are in direct analogy with an oriented surface's hyperelliptic mapping class group and its hyperelliptic Torelli subgroup, respectively; moreover, they enjoy the same relationship with $\calT_n$ as these groups of mapping classes do with the moduli space of homology-marked Riemann curves. 

A basic example of a member of $\STn$ is a \emph{doubled commutator transvection}. This is an automorphism that, for some $a, b$ and $c$ in a hyperelliptic basis $Y$, maps $c$ to \[ [a,b] \cdot c \cdot [b^{-1}, a^{-1}], \] and fixes the other members of $Y$. We prove the following theorem, establishing that such automorphisms suffice to generate $\STn$.

\begin{maintheorem}\label{hyptor} The group $\STn$ is generated by the set of doubled commutator transvections.
\end{maintheorem}
Theorem~\ref{hyptor} builds upon work of the second author \cite{Ful15}, which found a similar generating set for a subgroup of $\Autfn$ related to $\STn$ called the \emph{palindromic Torelli group}. These two related groups share the same definition, but in different settings: both centralize $\iota$, but members of $\STn$ need only do so up to an inner automorphism of $F_n$. 

The group $\STn$ is isomorphic to the fundamental group of each connected component of the locus of hyperelliptic graphs in $\mathcal{T}_n$; indeed, each component is an Eilenberg--Maclane space for $\STn$. The generating set given by Theorem~\ref{hyptor} allows us to prove that the components of the hyperelliptic graph locus are simply-connected `at infinity'. Precisely, $CV_n$ embeds as a dense open subset of a simplicial complex $\overl{CV}_n$ arising alternatively as Hatcher's sphere complex \cite{Hat95} or the free-splitting complex \cite{HM13}. This simplicial description descends to Torelli space, where each component of the hyperelliptic locus may be completed to a simplicial complex by adjoining points corresponding to degenerate hyperelliptic graphs (that is, hyperelliptic graphs of rank strictly less than $n$). We prove the following.

\begin{maintheorem} \label{locusconn}The simplicial completion of each component of the locus of hyperelliptic graphs in Torelli space is simply-connected. \end{maintheorem}

This theorem compares favorably with a theorem of Brendle--Margalit--Putman \cite[Theorem B]{BMP15}, who showed that the hyperelliptic locus in the space of homology-marked Riemann surfaces (that is, the branch locus of the classical period mapping) becomes simply-connected once surfaces of compact type are adjoined. Theorem~\ref{locusconn} also positively answers a question of Margalit, who asked if a free group analog of \cite[Theorem B]{BMP15} is true. In connection with tropical geometry, Chan--Galatius--Payne recently proved similar results about the connectivity of moduli spaces of graphs with marked points \cite{CGP16}.  

\textbf{Outline of paper.} In Section 2, we give the definition of Culler--Vogtmann's outer space of marked, metric graphs, and explain how to associate an inner product to such a graph to define the period mapping $\Phi$. The goal of Section 3 is to give a decomposition of a marked graph $\G$ along separating vertices and pairs of separating edges, which controls the fiber of the map $\Phi$ that $\G$ belongs to. This decomposition is considered further in Section 4, when we develop general machinery to describe how this decomposition may be altered, while remaining in a given fiber. In Section 5, we prove Theorem A. Turning to group theory, the aims of Sections 6 and 7, respectively, is to find generating sets for the groups $\HOut$ and $\STn$, proving Theorem B. Finally, in Section 8, we discuss the simplicial completion of outer space, in order to prove Theorem C.

\textbf{Conventions.} Functions are composed from right to left. For $g$ and $h$ in some group $G$, their commutator is denoted by $[g,h]$ and taken to be $ghg^{-1} h^{-1}$. The free group $F_n$ will have a preferred basis, which we denote by $X := \{ x_1, \dots, x_n \}$. The symmetric group on $X$ will be denoted $\Omega_n$. Given an automorphism $\theta \in \Autfn$, we denote its outer automorphism class as $[ \theta] \in \Outfn$. A \emph{cycle} in a graph $\Gamma$ is a closed, oriented path, and will be thought of as a sum of oriented edges in $C_1(\G)$. By assumption, cycles will be immersed (that is, may only cross over any edge at most once).

\textbf{Acknowledgements.} Both authors are grateful to Tara Brendle and Andrew Putman for their encouragement and guidance, and to Dan Margalit, whose question was the genesis of this project. They also thank Ruth Charney and Karen Vogtmann for helpful conversations, and are indebted to James Griffin for several clarifying conversations regarding a previous draft of this paper. Part of this work was completed while the authors were in residence at the Mathematical Sciences Research Institute in Berkeley, California, during the Fall 2016 semester.  


\section{The period mapping on Torelli space}

In this section we define the relevant spaces and maps we will be interested in for the remainder of the paper. Here we introduce outer space, Torelli space, the moduli space of graphs, the space of positive definite quadratic forms, and the moduli space of flat tori, as well as the various maps between them.  
\subsection{Outer space}

Let $CV_n$ denote Culler--Vogtmann `outer space'.  This is the space of marked metric graphs of volume 1 (without valence one or two vertices) whose fundamental group is $F_n$ the free group of rank $n$ \cite{CV86}.  This space admits a faithful, properly discontinuous action of the outer automorphism group $\Out(F_n)$, with quotient $\MG_n$, the moduli space of normalized metric graphs of rank $n$ without valence one or two vertices.  

It is well-known that the abelianization map $F_n\rightarrow \Z^n$ induces a surjection  $\Out(F_n)\rightarrow \GL_n(\Z)$, giving rise to a short exact sequence:\[1\rightarrow \calI_n\rightarrow \Out(F_n)\rightarrow \GL_n(\Z)\rightarrow 1,\] where the kernel $\calI_n$ is the so-called Torelli subgroup, or the group of (outer) automorphisms of $F_n$ that induce the identity on $H_1(F_n)\cong \Z^n$.  We will refer to the quotient of the action of $\calI_n$ on $CV_n$ as \emph{Torelli space} and denote it $\calT_n$.  Since $\calI_n$ is torsion-free, its action on $CV_n$ is free; hence the map $CV_n\rightarrow \calT_n$ is a covering and we can identify the fundamental group $\pi_1(\calT_n)$ with $\calI_n$.

Outer space $CV_n$ deformation retracts onto a subspace of smaller dimension called its \emph{spine}, which we denote by $K_n$. The spine may be thought of as the geometric realization of a poset of marked (combinatorial) graphs, with the ordering being determined by collapsing trees inside graphs.


\subsection{Positive definite quadratic forms}
We denote by $\calQ_n$ the symmetric space $\SL_n(\R)/\SO_n(\R)$ of noncompact type.  We can identify this space with the space of positive definite quadratic forms of rank $n$ by the spectral theorem via the mapping $A\in \SL_n(\R)$, $A\mapsto A^tA$, where $A^t$ denotes the transpose of $A$. An alternative description of $\calQ_n$ is as the space of marked lattices $\Z^n\hookrightarrow \R^n$, and hence the space of marked, flat $n$-dimensional tori.

There is an action of $\GL_n(\Z)$ on $\calQ_n$, given as follows: if $M\in \GL_n(\Z)$ and $P\in \calQ_n$ then $M \cdot P=M^tPM$.  This action corresponds to an integral change of basis for the lattice represented by $P$, and hence descends to an isometry between the two marked flat tori represented by $P$ and $M \cdot P$.  Conversely, any isometry between two flat tori induces an integral change of basis between their lattices hence is given by an element of $\GL_n(\Z)$. Hence we can identify the quotient $\calQ_n/ \GL_n(\Z)$ with the moduli space of flat $n$-dimensional tori up to isometry, which we denote by $\MT_n$.  Note that this action by $\GLn$ is not faithful; the kernel is the subgroup $\{\pm I\}$.

\subsection{The period mapping}
Associated to any metric graph $\Gamma$ is a certain inner product defined on 1-simplices as follows.  Let $e_1,\ldots, e_m$ denote the edges of $\G$, each with a chosen orientation, so that they form a basis for $C_1(\Gamma)$. Define \[\langle e_i, e_j\rangle_\Gamma=l(e_i)\cdot \delta_{ij},\]
where $l(e_i)$ is the length of the edge $e_i$ in $\Gamma$.  Extend this by linearity to the rest of $C_1(\Gamma)$. This inner product on $C_1(\Gamma)$ descends to an inner product on $H_1(\Gamma)\cong \Z^n$, which is clearly positive definite.  

Now suppose we are given a marking $\rho: R_n\rightarrow \Gamma$ where $\rho$ is a homotopy equivalence and $R_n$ is the wedge of $n$ circles. Order and orient these circles, and denote the loops they represent in $\pi_1(R_n)$ by $x_1, \dots , x_n$. The map $\rho$ also gives a homology marking for $\Gamma$ and we can define a map $\widehat{\Phi}:CV_n \rightarrow \calQ_n$ using the above inner product on $C_1(\Gamma)$:\[\widehat{\Phi}(\Gamma,\rho)=\frac{1}{\det(M)}M=\frac{1}{\det(M)}(m_{ij}), \mbox{ where $m_{ij}=\langle\rho_*(x_i),\rho_*(x_j)\rangle_\Gamma$}.\]

Since the action of the Torelli subgroup does not change the homology marking induced by $\rho$, it follows that $\widehat{\Phi}$ factors through Torelli space to give a map $\Phi:\calT_n\rightarrow \calQ_n$.  Moreover, it is not hard to see that $\Phi$ is equivariant with respect to the $\GL_n(\Z)$ actions on $\calT_n$ and $\calQ_n$ defined above, hence descends to a map $\overl{\Phi}:\MG_n\rightarrow \MT_n$. We will refer to $\overl{\Phi}$ as the \emph{Abel--Jacobi map} for rank $n$ graphs, since it is analogous to the classical Abel--Jacobi map for a Riemann surface. In this paper we will be most interested in understanding the middle map $\Phi$ and in particular its fibers.  We end the section by summarizing all of the spaces and maps just defined in the following commutative diagram.\[ \xymatrixrowsep{5pc}\xymatrixcolsep{5pc} \xymatrix{
  CV_n\ar[d]^{\cdot/\calI_n} \ar[dr]^{\widehat{\Phi}}  \\
\calT_n\ar[d]^{\cdot/\GL_n(\Z)}\ar[r]^{\Phi} &  \calQ_n\ar[d]^{\cdot/\GL_n(\Z)}\\
\MG_n\ar[r]^{\overl{\Phi}} & \MT_n
  }
\]

Since $\Phi$ is $\GLn$-equivariant, to determine the fiber containing a marked graph $(\Gamma, \rho)$, we may make a strategic change of the homology marking on $\Gamma$. This will be our approach in the following section.


\section{A Canonical Decomposition of a Graph}\label{Decomp}
Let $\Gamma=(V,E)$ be a graph with vertex set $V$ and edge set $E$. Moving around in outer space corresponds to blowing up and collapsing forests in marked rank $n$ graphs.  In this section we will be interested in what operations on graphs (blowing up edges, collapsing edges) do not alter the inner product on the graph's homology. 
\subsection{Separating edges}
As a first observation, the following proposition tells us that separating edges do not affect the image of $\Phi$. 

\begin{proposition} Let $(\Gamma,\rho)$ be a homology-marked graph with a separating edge $e$, and let $\pi:\Gamma\rightarrow\Gamma/e$ be the map which collapses $e$ to a point.  Then $\Phi(\Gamma,\rho)=\Phi(\Gamma / e ,\pi\circ\rho)$.   
\end{proposition} 
\begin{proof} Since $e$ is separating, $\pm e$ does not occur in any cycle on $\Gamma$. Thus, up to rescaling the metric on $\Gamma / e$ to be volume 1, the inner products on the given homology bases for the graphs are indistinguishable. \end{proof}

As a result of this proposition we see that $\widehat{\Phi}$ actually factors through so-called \emph{reduced} outer space. This is a subspace ${CV_n}'$ consisting of those marked graphs without separating edges, onto which $CV_n$ strongly deformation retracts. For any graph $\Gamma$, the deformation retraction shrinks all of the separating edges of $\Gamma$ and simultaneously scales each of the other edges by a constant amount so that the volume 1 condition is preserved.  Note that ${CV_n}'$ is preserved by the action of $\Out(F_n)$, so in light of the above proposition, from now on we assume that $\Gamma$ has no separating edges.

\subsection{Separating vertices}
If $v\in V$ is a vertex, consider $\Gamma\setminus \{v\}=\coprod_{i=1}^k\Gamma_i$, where each $\Gamma_i$ is a connected component of $\Gamma\setminus \{v\}$.  For each $\Gamma_i$ we can consider the inclusion $d :\Gamma_i\rightarrow \Gamma$, and let $\overl{\Gamma_i}$ denote the topological closure of $d(\Gamma_i)$ (that is, the result of gluing the vertex $v$ back into $d(\Gamma_i)$). Then we define $\Gamma$ \emph{cut along} $v$ to be the graph $\Gamma||v := \coprod_{i=1}^k\overl{\Gamma_i}$. In particular, if $v\in V$ does not separate, then $\Gamma||v\cong \Gamma$. We write $\G || (v_1 \cup \dots \cup v_m)$ to mean the graph obtained by first cutting along $v_1$, then cutting the result along $v_2$, and so on, until $v_m$. We prove now that this process terminates in a unique way.

\begin{proposition}\label{maxsep} $\Gamma$ has a maximal decomposition along separating vertices, unique up to reordering components.  
\end{proposition}
The proof of this proposition can be deduced from the following easy lemma.

\begin{lemma}\label{sepgone}Let $v$ be a vertex of $\Gamma$, and let $\overl{\Gamma_1}$ be some component of $\Gamma||v$.  If $v_1$ is the vertex which corresponds to $v$ in $\overl{\Gamma_1}$, then $v_1$ does not separate $\overl{\Gamma_1}$.  In other words, $(\Gamma||v)||v_1\cong\Gamma||v$.

\end{lemma}
\begin{proof} Let $e_1,\cdots, e_r$ be the edges of $\Gamma$ incident at $v$ that all belong to the connected component $\Gamma_1$.  The midpoint of each of these edges can be connected in $\Gamma\setminus\{v\}$ by a path lying entirely in $\Gamma_1$.  Each of these edges are in one to one correspondence with the edges incident at $v_1$ in $\overl{\Gamma_1}$, so the lemma follows.  
\end{proof}

We now finish the proof of the proposition.  
\begin{proof}[Proof of Proposition~\ref{maxsep}] From Lemma~\ref{sepgone}, it follows that if $v_1,\ldots,v_m$ are the separating vertices of $\Gamma$, then $(\dots(\Gamma||v_1)||v_2)\dots)||v_m$ has no separating vertices, since cutting along $v_1, \dots, v_m$ certainly does not produce any new separating vertices in the components. Obviously we have $(\Gamma||v)||v'\cong(\Gamma||v')||v$, so the order in which we cut separating vertices does not matter.  
\end{proof}

Define the \emph{component graph} $\mathcal{C}(\Gamma)$ to be the metric graph obtained by cutting $\Gamma$ at all of its separating vertices (that is, $\mathcal{C}(\Gamma) = \coprod_{v \in S} \Gamma || v$, where $S$ is the set of separating vertices in $\G$).  It may happen that some vertices made when forming $\C(\G)$ have valence two.  We forget these vertices when we define $\mathcal{C}(\Gamma)$, allowing the two edges incident at this vertex to become a single edge. We can think of $\Gamma$ as an iterated wedge of the components of $\mathcal{C}(\Gamma)$. As such we know that $H_1(\Gamma)=\bigoplus_{i=1}^mH_1(\tilde \Gamma_i)$, where $\tilde \Gamma_i$ are the components of $\mathcal{C}(\Gamma)$.  Then for any homology marking $\rho$ of $\Gamma$, there is an equivalent marking under the $\GL_n(\Z)$ action image under $\Phi$ that is block diagonal.  This observation leads us to the following proposition.  

\begin{proposition}\label{splittingblow-up} Let $v$ be a separating vertex of a homology-marked graph $(\Gamma,\rho)$.  Suppose $(\Gamma',\rho')$ is the result of blowing up a tree at $v$. Then $\Phi(\Gamma,\rho)=\Phi(\Gamma',\rho')$ if and only if $\C(\Gamma')$ is a blow-up of $\mathcal{C}(\Gamma)$.  
\end{proposition}
\begin{proof}

The reverse implication is immediate, by choosing a block diagonal marking of $\Gamma$ as in the discussion preceding the statement of the proposition. 

To prove the forward implication, let $\Gamma_1,\ldots, \Gamma_r$ be the components of $\Gamma\setminus \{v\}$, and denote by $T$ the tree blown up at $v$ to form $\Gamma'$. Note that $\Gamma_i$ sits inside $\G '$ as a subspace; we denote its closure by $\overl{\G_i}$. For each component $\Gamma_i$, we obtain a subtree $T_i$ of $T$ consisting of all arcs in $T$ connecting the vertices of $\overl{\Gamma_i} \cap T$.  Observe that $T=\cup_i T_i$, since, if not, $\Gamma '$ would have a separating edge.  

To prove the proposition it suffices to show that for $i\neq j$, the graph $T_i\cap T_j$ is either empty or a single vertex (since this gives that $T$ may be blown up incrementally, inside one cut component of $\G$ at a time).  Suppose that, in fact, some $T_i$ and $T_j$ share an arc $b$. Let $b_i$ be an arc in $T_i$ containing $b$ and connecting two edges $e_1$ and $e_2$ of $\Gamma_i$.  Similarly, find $b_j$ an arc in $T_j$ connecting edges $f_1$ and $f_2$ of $\Gamma_j$ and containing $b$. Since $\Gamma_i$, $\Gamma_j$ are connected, $b_i$ and $b_j$ can be completed to non-trivial, primitive cycles $\alpha_i$ and $\alpha_j$ so that $\alpha_i$ contains $e_1$ and $e_2$ and $\alpha_j$ contains $f_1$ and $f_2$. Such a configuration can be seen in Figure~\ref{InsertEdge}.

After changing our basis perhaps, we can assume that $\alpha_i$ and $\alpha_j$ are part of the homology basis for $\Gamma'$, and their images under the collapse of $T$ are such for $\Gamma$.  Then in $\Gamma'$ we have that $\alpha_i\cdot\alpha_j=l(b)$, while in $\Gamma$ their images have trivial inner product.  This contradicts our premise that $\Phi(\Gamma,\rho)=\Phi(\Gamma',\rho')$.
\end{proof}

\begin{figure}[h]
\begin{center}
\begin{overpic}[width=2in]{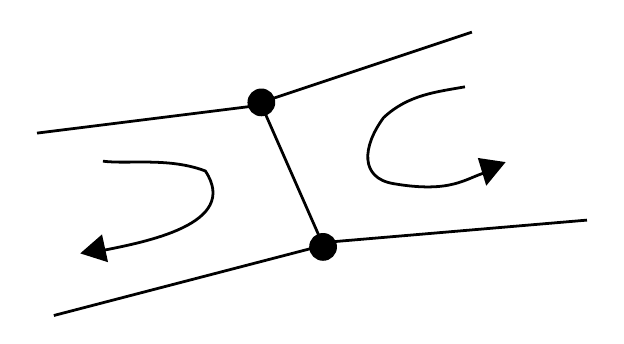}
\put(20,40){$e_1$}
\put(30,4){$e_2$}
\put(55,48){$f_1$}
\put(70,10){$f_2$}
\put(40,24){$b$}
\end{overpic}
\caption{Inserting an edge between two subgraphs $\Gamma_1$ and $\Gamma_2$.  The path $e_1be_2$ belongs to a cycle $\gamma_1\subset\Gamma_1$, and the path $f_1bf_2$ belongs to a cycle $\gamma_2\subset\Gamma_2$.}
\label{InsertEdge}
\end{center}
\end{figure}
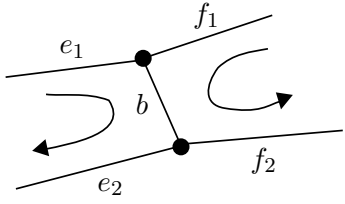

Given $\Gamma$, consider $\mathcal{C}(\Gamma)=\coprod_{i=1}^k \tilde \Gamma_i$ and fix a homology marking of each $\tilde \G_i$. These induce a homology marking $\rho$ of $\Gamma$ that is block diagonal. Moreover, any wedge product of the components of $\C (\G)$ will also inherit a homology marking, and will lie in the same fiber of $\Phi$ as $(\G, \rho)$. All such products lie in the same connected component of the fiber, as the choice of wedge points may be varied by altering edge lengths in $\calT_n$. Proposition~\ref{splittingblow-up} implies that any single edge blow up of $\Gamma$ at a separating vertex may be achieved by carrying out a different wedging of the components of $\C(\G)$.

\subsection{Pairwise separating edges}
We now assume that $\Gamma$ does not have separating vertices. In practice, this is because we will be working inside the components of $\C(\G)$. Our goal will be to decompose $\Gamma$ further along collections of pairwise separating edges.  

Let $e$ and $f$ be edges of $\Gamma$.  We write $e\sim f$ if $e\cup f$ separates $\Gamma$ (necessarily into two components), and call $e$ and $f$ a \emph{separating edge pair}. Alternatively, we may say that $e$ and $f$ \emph{pairwise separate}. Obviously $e\sim f$ implies $f\sim e$.  In fact, this relation is transitive.  

\begin{lemma}\label{Equivalence}Let $e, f$ and $g$ be edges of $\Gamma$. If $e\sim f$ and $f\sim g$ then $e\sim g$.  
\end{lemma}
\begin{proof} Since both $e$ and $f$, and $f$ and $g$ are separating edge pairs, we can find a cycle $\gamma$ in $\Gamma$ that travels over all three edges. Assume then for contradiction that $e\sim f$ and $f\sim g$ but $e\nsim g$. Without loss of generality, $e$, $f$, and $g$ appear clockwise around $\gamma$ in alphabetical order.  Let the three complementary regions $\gamma\setminus (e\cup f\cup g)$ be $\alpha_{ef}$, $\alpha_{fg}$ and $\alpha_{ge}$, respectively. Since $e\cup g$ does not separate, this means either there is a path in $\Gamma\setminus (e\cup g)$ from $\alpha_{ef}$ to $\alpha_{eg}$ or from $\alpha_{fg}$ to $\alpha_{eg}$.  This implies that in the case of the former, $e\cup f$ does not separate, while in the latter $f\cup g$ does not separate.  This contradiction shows that $e\cup g$ must separate.  
\end{proof}

Let $[e_1],\ldots, [e_r]$ denote the equivalence classes of edges generated by pairwise separation in $\Gamma$. Denote by $\overl{[e_i]}$ the subgraph of $\G$ consisting of all the edges in $[e_i]$ together with their endpoints. Each edge of $\overl{[e_i]}$ has two distinct vertices, and none are shared between edges, since any shared vertex would be separating in $\G$.  Moreover, we have the following lemma.

\begin{lemma}\label{Forest}Suppose $\Gamma$ does not have separating vertices. The subgraph $F_\G=\overl{[e_1]}\cup\cdots\cup\overl{[e_m]}$ is a forest.  
\end{lemma}
\begin{proof}Let $\gamma$ be a cycle whose edges are all part of $F_\Gamma$. Then if $e\subset \gamma$ and $e\sim f$, we know that $f\subset \gamma$, since the cycle $\gamma$ must travel between the two components of $\G \setminus \{ e, f \}$. Since $\G$ has no separating vertices, no two adjacent edges in the cycle $\gamma$ may lie in the same subgraph $\overl{[e_i]}$. This, together with the fact that $\Gamma$ has no separating edges, means that $\gamma$ must traverse at least four edges. We may thus assume that in the cyclic order on the edges of $\gamma$, there is a subsequence of edges either of the form $e f e' f'$ or $e e' f f'$, where $[e] = [e'] \neq [f] = [f']$.

Suppose that $\gamma$ contains the subsequence $e f e' f'$. Consider the components $C_1$ and $C_2$ of $F_\G \setminus \{e, e' \}$. Suppose the edge $f$ lies in $C_1$. If $C_1$ was separated by $f$, this would force $e$ and $f$ to pairwise separate $F_\G$. Thus, there is a path $p_1$ in $C_1$ (and, similarly, $p_2$ in $C_2$) joining the endpoints of $e$ and $e'$. 

Now, the assumed cyclic ordering on $\gamma$ forces $e$ and $e'$ to belong to different components of $\G \setminus \{ f, f' \}$. However, the paths $p_1$ and $p_2$ allow us to connect these components, without passing through $f$ or $f'$. This contradicts that $f$ and $f'$ form a separating edge pair, hence this cyclic ordering on $\gamma$ cannot occur.

Finally, suppose that $\gamma$ contains the subsequence $e e' f f'$. We will such that such a $\gamma$ must contain some other subsequence of the form considered in the previous paragraph, and hence the proof will be complete. Suppose not, for the sake of contradiction. Let $I_1 , \dots , I_\ell$ denote the arcs obtained by removing the edges of $[e]$ from $\gamma$. Then, for each $I_j$, the intersection of $I_j$ with each subgraph $\overl{[e_i]}$ is either edgeless or all of $\overl{[e_i]}$. Continuing to excise separating class in this fashion, we eventually exhaust either the edges of $\gamma$ or the classes of separating edges, and a contradiction is forced.
\end{proof}

Knowing now that the graph $F_\G$ is a forest, we proceed to show that blow-ups and blow-downs of such forests are the only marked graphs near $(\G, \rho) \in \calT_n$ that lie in its fiber under $\Phi$. 

Let $[e]=\{e=e_1,\ldots,e_r\}$ be some class of pairwise separating edges in $\Gamma$.  Then $[e]$ defines an $(r-1)$-simplex in the connected component of the fiber containing any marking $(\Gamma,\rho)$ as follows.  Note that any embedded cycle that meets some edge in $[e]$ meets all the others (otherwise some pair of edges would not separate). Moreover, any two such embedded cycles either traverse all the edges of $[e]$ in the same direction as each other, or all in the opposite direction. It follows that if $\sum_{i=1}^r l(e_i)=L$, any variation of the lengths of the edges of $[e]$ that does not change the total length $L$ will not change the image under $\Phi$. The next proposition says that in a graph without separating edges, this is the only edge expansion which can be done within the fibers of $\Phi$. 

\begin{proposition} \label{SepPair}Let $(\Gamma,\rho)$ be a homology marked graph, and let $(\Gamma',\rho')$ result from blowing up a forest $F$ in $\Gamma$.  Then  $\Phi(\Gamma,\rho) = \Phi(\Gamma',\rho')$ if and only if $F$ is a union of separating edge pairs in $\G '$.  
\end{proposition}
\begin{proof}The reverse implication follows from the discussion in the paragraph preceding the statement of the proposition. The main part of the proof involves establishing the forward implication.

Consider some edge $b$ contained in a tree $T \subset F \subset \G '$. We assume that $b$ is not part of a separating pair of edges, in order to derive a contradiction. Our goal is to find (part of) a homology marking for $\G'$ that evaluates differently under $\Phi$ than the marking of $\Gamma$ that collapsing $F$ induces. We make the following observation:

$(*)$ If $\alpha$ and $\beta$ are cycles in $\Gamma'$ such that $\alpha\cap \beta\subset F$, then $\Phi(\Gamma,\rho)\neq \Phi(\Gamma',\rho')$.

For then in $\Gamma'$, the intersection pairing gives $\langle\alpha,\beta\rangle_{\Gamma'}\neq0$, while the corresponding cycles in $\Gamma$ have trivial inner product.  

To apply this observation, select two distinct embedded cycles $\alpha$ and $\beta$ that traverse the edge $b$ in the same direction. We will repeatedly modify $\alpha$ and $\beta$ until we obtain two embedded cycles containing $b$ and whose intersection lies entirely in the forest $F$. The fact that $b$ is not separating means that there exist cycles containing $b$, and the assumption that $b$ is not part of a separating edge pair implies that there are at least two distinct such cycles. 

\begin{claim} \label{Orient}We can choose $\alpha$ and $\beta$ so they both traverse the edges of $\alpha\cap\beta$ in the same direction.

\end{claim} 
Indeed, $\alpha$ and $\beta$ overlap in some number of intervals.  Write $\alpha=\alpha_1\delta\alpha_2$ and $\beta=\beta_1\overl{\delta}\beta_2$, for some paths $\delta, \alpha_i, \beta_i$ in $\G'$ ($i=1,2$), where $\alpha_1$ and $\beta_1$ both begin with $b$. Then $\alpha$ and $\beta$ cross $\delta$ with opposite orientation.  Note that $\delta$ cannot travel over $b$ by assumption.  Replace $\alpha$ by $\alpha'=\alpha_1\beta_2$ and $\beta$ by $\beta'=\beta_1\alpha_2$. Then $\alpha'$ and $\beta'$ both contain $b$ and have one fewer interval where they overlap with opposite orientation.  

After applying the above claim, we will have two distinct cycles that may intersect in some intervals, but traveling along them in the same direction. Now, think of $\alpha$ as an interval $I_\alpha=[0,l(\alpha)]$ with endpoints identified and $\beta$ as an interval $I_\beta=[0,l(\beta)].$ Label the intervals of intersection in $I_\beta$ from 1 to $n$, increasing in the positive direction of $[0,l(\beta)]$, and label the corresponding intervals $I_\alpha$ similarly. (We may choose any vertex on $\alpha$ and $\beta$ as the `0' of the intervals $I_\alpha$ and $I_\beta$). Thus the ordering of the intervals in $I_\alpha$ defines a permutation $\sigma$ of $\{1,\ldots, n\}$.

\begin{claim} \label{Permute} We can choose $\alpha$ and $\beta$ so that $\sigma$ is the identity permutation.
\end{claim}
To prove this claim, suppose that for two integers $m_1$ and $m_2$ with $m_1<m_2$ we have that interval $m_2$ occurs before $m_1$ on $I_\alpha$.  We can write $\alpha=\alpha_1j_{m_2}uj_{m_1}\alpha_2$ and $\beta=\beta_1j_{m_1}vj_{m_2}\beta_2$, for paths $u, v, j_{m_i}, \alpha_i, \beta_i$ in $\G '$ ($i =1,2$).  Then replace $\alpha$ by $\alpha'=\alpha_1j_{m_2}\beta_2$ and $\beta$ by $\beta'=\beta_1j_{m_1}\alpha_2$.  This avoids all the intervals in between $j_{m_1}$ and $j_{m_2}$ inclusive. Repeatedly applying this process, we remove any intervals of intersection that appear `out of order' in $\alpha$.

At this point, $\alpha$ and $\beta$ may overlap in some number of intervals, but in the same cyclic order and with the same orientation.  The last step is to ensure that the intersections all lie in the forest $F$. 

\begin{claim}\label{Intersect}It is possible to replace $\alpha$ and $\beta$ by distinct cycles whose only intersection lies in $F$.  
\end{claim}
To prove the claim, we induct on the number of edges in the intersection of $\alpha$ and $\beta$ not lying in $F$.  Choose some edge $e\in \alpha\cap\beta$. Observe that, given the result of Claim 2, there are exactly two components $C_1$ and $C_2$ of $\alpha\cup\beta\setminus\{b,e\}$, one of which is either a point or a cycle itself.  By assumption $\Gamma'\setminus\{b,e\}$ is connected.  Thus, we can find a path $\nu$ from $C_1$ to $C_2$. Depending on whether the endpoints of $\nu$ lie on $\alpha$ or $\beta$ (or a combination), there are several cases to consider, but they are all similar.  We will consider the case where $\gamma$ connects a point in an arc belonging to $\alpha$ to a point in an arc belonging to $\beta$.  Write $\alpha=\alpha_1\alpha_2e\alpha_3$ where the endpoint of $\alpha_1$ is the start of $\gamma$, and $\beta=\beta_1e\beta_2\beta_3$ where the initial point of $\beta_2$ is the endpoint of $\gamma$.  Then set $\alpha'=\alpha_1\gamma\beta_2$ and $\beta'=\beta_1e\alpha_3$.  We have successfully avoided the intersection at $e$, and the total number of intersections has decreased since $\gamma$ is disjoint from $\alpha\cup \beta$.  

From this final claim, we find cycles $\alpha$ and $\beta$ intersecting only in $F$; hence, by $(*)$ we conclude that $\Phi(\Gamma,\rho)\neq \Phi(\Gamma',\rho')$.
\end{proof}


We end this section by proving that the blow-ups described in Proposition~\ref{SepPair} result in a unique `maximal' blow-up. A graph is said to be \emph{maximally separating} if none of its blow-ups increase the number of edges that form separating edge pairs. 


\begin{lemma}\label{uniqueblow-up} Suppose $(\Gamma,\rho)$ has no separating vertices.  Then there exists a maximally separating blow-up $(\Gamma',\rho')$ such that $\Phi(\Gamma,\rho)=\Phi(\Gamma',\rho')$. Moreover, $\Gamma'$ is unique up up to (combinatorial) graph isomorphism. 
\end{lemma}
\begin{proof} First, observe that an edge in a separating pair is the unique edge having its endpoints.  It follows that an edge of a separating pair may be blown up at a vertex $v$ if and only if there exists an edge $e$ such $\Gamma \setminus \{e, v \}$ is disconnected, and in each component of $\Gamma\setminus \{e,v\}$ there are at least 2 edges which meet $v$.  There is a unique way to insert an edge $f$ between these two components. Blowing up this edge does not create any new edge-vertex pairs of this form. It follows that the isomorphism type of a maximally separating blow-up of $\Gamma$ depends only the vertex-edge pairs $(e,v)$ in $\Gamma$ described above. It is left as a straightforward exercise to verify that the order in which these vertices are blown up to edges does not matter, and so there is a unique way to blow up the vertices in such pairs into pairwise separating edges. 
\end{proof}


\section{The Configuration Space of a Wedge of Metric Spaces}\label{Wedge}
In this section, we develop machinery that lets us describe how the set of separating vertices in a marked graph $(\G, \rho)$ can vary, without leaving the fiber $\Phi^{-1}(\Phi (\G, \rho))$. It will be convenient to use a more general framework, and consider a wider class of metric spaces. We note that a basepointed version of this construction was developed previously by Griffin \cite{Gri13}, and similar spaces were considered by Collins \cite{Col89}.

Let $X_1,\ldots, X_m$ denote a collection of connected path-metric spaces, and let $X=\bigvee_{i=1}^nX_i$ be their wedge product. For our purposes, the $X_i$ will be simplicial complexes with their induced path metrics.  In this section we ask the following 
\begin{question}Up to isometries isotopic to the identity, what are all the different ways of wedging the $X_i$ together?
\end{question}

This question leads to a configuration space of all possible ways of wedging the $X_i$ together.  For instance, if $n=2$ this amounts to a choice of wedge point in $X_1$ and $X_2$, hence the configuration space is just $X_1\times X_2$.  

The question, however, specifies all wedges up to isometric isotopy.  To see the difference, again consider two spaces where $X_1=X_2=\Sa^1$ the circle of radius 1.  By parametrizing each circle, the configuration space of possible wedges is a torus $\Sa^1\times\Sa^1$.  But every such wedging is isometric to every other one, hence the space of wedges is a single point.  To obtain the space of wedges up to isometry, we could take the configuration space and identify points that ``look the same" under the action of the isometry group. Let $\Isom_0 (X_i)$ be the isometries of $X_i$ that are isotopic to the identity map. We obtain the space of wedges by taking the configuration space $X_1\times X_2=\Sa^1\times\Sa^1$ and taking the quotient by $\Isom_0(X_1)\times\Isom_0(X_2)=\SO_2(\R)\times \SO_2(\R)$. Since this group acts transitively, the quotient is a point, as desired.  This will be a motivating example for the rest of this section.

\subsection{Labelled trees and the configuration space of wedges}
Let $T=(V,E)$ be a labelled tree with $n$ vertices, with label set $\{1,\ldots,n\}$.   Given metric spaces $X_1,\ldots, X_m$, we form a space $\widetilde{T}(X_1,\ldots,X_m)$ as follows.  For each edge $e\in T$ with endpoints labelled $i,j$, there is a copy of $X_i\times X_j$ in $\widetilde{T}(X_1,\ldots,X_m)$.  In other words,\[\widetilde{T}(X_1,\ldots,X_m)=\prod_{(i,j)\in E}X_i\times X_j.\]

We call $\widetilde{T}(X_1,\ldots,X_m)$ the \emph{parametrized configuration space of wedges with pattern $T$}.  Obviously we have that after rearranging, $\widetilde{T}(X_1,\ldots,X_m)\cong X_1^{d_1}\times\cdots\times X_n^{d_n}$, where $d_i$ is the degree of the vertex labelled $i$ in $T$.  Therefore the group $\Isom_0(X_1)\times\cdots\times \Isom_0(X_m)$ acts diagonally on $\widetilde{T}(X_1,\ldots,X_m)$.  The quotient space $T(X_1,\ldots, X_m)$ by this action is what we call the \emph{configuration space of wedges} with pattern $T$.  

A well-known theorem of Cayley states that there are exactly $m^{m-2}$ labelled trees with $m$ vertices.  We will form the configuration space of wedges by gluing together the configuration spaces of all $m^{m-2}$ patterns.  Given labelled trees $T_1$ and $T_2$, we glue $\widetilde{T}_1 (X_1, \dots, X_m)$ and $\widetilde{T}_2 (X_1, \dots, X_m)$ together along factors $X_i \times X_j$ whenever both trees have an edge with vertices labelled $i$ and $j$. Since the action of $\Isom_0(X_1)\times\cdots\times \Isom_0(X_m)$ acts in the same way on each parametrized configuration space and preserves the subspaces where they agree, we can glue together the quotients $T_1(X_1,\ldots,X_m)$ and $T_2(X_1,\ldots,X_m)$.

We define the \emph{configuration space of wedges} to be $W(X_1,\ldots,X_m)=\coprod_{i=1}^{n^{n-2}}T_i(X_1,\ldots,X_m)/\sim$, where we identify to different patterns along the diagonal embeddings as above.
\subsection{Wedges of graphs}
In our situation, the $X_i$ will be the components of the splitting graph $\mathcal{C}(\Gamma)$ and, in particular, $\Isom_0(X_i)=1$ unless $X_i=S^1$ in which case $\Isom_0(X_i)\cong\Sa^1$. If $\mathcal{C}(\Gamma)=\Gamma_1\cup\cdots \cup \Gamma_k\cup S_1^1\cup\cdots\cup S^1_l$. If $T$ is a labelled $(k+l)$-vertex tree then we have \[T(\Gamma_1,\ldots,\Gamma_k,S_1^1,\ldots,S_l^1)=\frac{\Gamma_1^{d_1}\times\cdots\times \Gamma_k^{d_k}\times (S^1)^{d_{k+1}}}{(\Sa^1)^l},\] where $l\leq d_{k+1}$.  The main result of this section is the following.

\begin{proposition}\label{Config} Suppose that $\Gamma_i'$ is obtained from $\Gamma_i$ by a forest collapse, for each $1 \leq i \leq m$. Then \[W(\Gamma_1,\ldots,\Gamma_m)\simeq_{h.e.}W(\Gamma_1',\ldots,\Gamma_m'). \]
\end{proposition}

\begin{proof}Since the quotient of a torus by a subtorus is again a torus, the computation above shows that for each pattern $T$, we have that $T(\Gamma_1,\ldots,\Gamma_m)$ is a product of graphs and a torus of some dimension. 

Now, note that if $\Gamma_1'$ is obtained from $\Gamma_1$ by a forest collapse, then neither one is a circle, hence $\Isom_0(\Gamma_1')=\Isom_0(\Gamma_1)=1$.  Then for any pattern $T$ we have 
\begin{align*}T(\Gamma_1,\ldots,\Gamma_m)&=\frac{\Gamma_1^{d_1}\times\cdots\times \G_m^{d_m}}{\Isom_0(\G_1)\times\cdots\times \Isom_0(\G_m)}\\
&= \G_1^{d_1}\times\left(\frac{\Gamma_1^{d_2}\times\cdots\times \G_m^{d_m}}{\Isom_0(\G_2)\times\cdots\times \Isom_0(\G_m)}\right)\\
&\simeq_{h.e.}\G_1'^{d_1}\times\left(\frac{\Gamma_1^{d_2}\times\cdots\times \G_m^{d_m}}{\Isom_0(\G_2)\times\cdots\times \Isom_0(\G_m)}\right)\\
&\simeq \dots \simeq T(\Gamma_1',\ldots,\Gamma_m'),
\end{align*}
where the homotopy equivalence is obtained by collapsing componentwise. Since this process commutes with the gluing maps between different patterns, we obtain a homotopy equivalence $W(\Gamma_1,\ldots,\Gamma_m)\simeq W(\Gamma_1',\ldots,\Gamma_m')$ as desired. 
\end{proof}

\subsection{Asphericity of the configuration space of wedges}

To close this section, we generalize work of Collins \cite{Col89}, showing that the configuration space $W(\Gamma_1,\ldots,\Gamma_m)$ is aspherical and identifying its fundamental group.

\textbf{Colored symmetric graphs.} Writing $\mathcal{C}(\Gamma)=\coprod_{i=1}^k \Gamma_i$, the configuration space of wedges associated to $\Gamma$ is $W=W(\Gamma_1,\ldots, \Gamma_k)$.  First, we replace $W$ with a space to which it is homotopy equivalent.  Choose a maximal tree $T_i$ in $\Gamma_i$.  By Proposition~\ref{Config}, we know that $W$ is homotopy equivalent to $W'=W(G_1/T_1,\ldots, G_k/T_k)$. The space $W'$ may be described as follows.  Each $G_i/T_i=R_{\ell_i}$ is a rose of rank $\ell_i$ so that $\sum_{i=1}^k \ell_i=n$, and we think of the rose $R_n$ as $R_{\ell_1}\vee\cdots\vee R_{\ell_k}$, where each rose $R_{l_i}$ has a different color. Call this decomposition a \emph{colored wedge of roses}. In $W'$, we only allow blow-ups in which each rose of a given color stays wedged at the same point. We call such a blow-up a \emph{colored symmetric graph}; see Figure~\ref{coloredrose} for an example. Roses of different colors may move along edges of a different color at their wedge point, but loops of the same color must always stay bunched. Note that if $\G=R_n$, then each loop has a different color, and in this case we just call the corresponding blow-ups \emph{symmetric}.
\begin{figure}[h]
\centering
\includegraphics[width=2.5in]{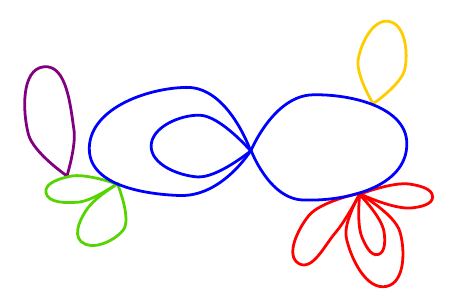}
\caption{A colored symmetric graph.}
\label{coloredrose}
\end{figure}

\textbf{Pure colored symmetric automorphisms.} Any colored wedge of roses gives us a free splitting $F_n=F_{\ell_1}*\cdots*F_{\ell_k}$, where $\sum_{i=1}^k \ell_i=n$.  The \emph{pure colored symmetric automorphism group} associated to this free splitting, denoted $\mathrm{P}\Sigma(\ell_1,\ldots, \ell_k)$, is the subgroup in $\Out(F_n)$ generated by all the partial conjugations of the form \[F_{\ell_i}\mapsto xF_{\ell_i}x^{-1}, \text{ where } x\in F_{\ell_j}, 1\leq j \neq i \leq k\] and no other free factor is conjugated by $x$. The notation `$\mathrm{P}\Sigma$' here stands for `pure symmetric',  and comes from the case when each $\ell_i=1$. In this instance, the pure colored symmetric automorphism group equals $\mathrm{P}\Sigma_n$, the \emph{pure symmetric automorphism group}. The group $\mathrm{P} \Sigma_n$ can also be thought of as the subgroup of $\Out(F_n)$ consisting of automorphisms that preserve the conjugacy class of each generator in a chosen basis. 

Collins studied $\mathrm{P}\Sigma_n$ in \cite{Col89}, and showed that the moduli space of symmetric graphs is aspherical with fundamental group $\mathrm{P}\Sigma_n$.  In our language, Collins proved that any connected component of a fiber of $\Phi$ containing a rose is aspherical and $\pi_1$-injective, with fundamental group $\mathrm{P}\Sigma_n$. We will extend Collins' result to all colored symmetric graphs:

\begin{theorem} \label{ColoredCollins}For any metric on $R_{\ell_i}$, the configuration space $W(R_{\ell_1},\ldots,R_{\ell_k})$ is aspherical with fundamental group $\mathrm{P}\Sigma(\ell_1,\ldots, \ell_k)$. \end{theorem}

Our proof of Theorem \ref{ColoredCollins} will closely follow Collins' proof, and has three major steps.  Recal that $K_n$ denotes the spine of outer space $CV_n$.  First, we show that the union $K_{\ell_1,\ldots, \ell_k}$ of stars of roses that are marked by elements of $\mathrm{P}\Sigma(\ell_1,\ldots, \ell_k)$ is a connected,  contractible subspace of $K_n$. Second, we use a proposition of Collins to show that $K_{\ell_1,\ldots, \ell_k}$ deformation retracts onto the subspace $(K_{\ell_1,\ldots, \ell_k})^\Sigma$ consisting of all symmetric graphs contained in $K_{\ell_1,\ldots, \ell_k}$.  In the final step, we show that $(K_{\ell_1,\ldots, \ell_k})^\Sigma$ deformation retracts onto the subspace of marked, colored symmetric graphs, denoted $\mathrm{C}\Sigma(R_{\ell_1},\ldots, R_{\ell_k})$.   

\textbf{Minimal posets.}
Let $\mathcal{W}$ be a set of conjugacy classes of elements of $F_n$.  Then $\mathcal{W}$ defines a norm on marked roses $(R_n,\rho)\in K_n$ via $\|(R_n,\rho)\|_{\mathcal{W}}=\sum_{w\in \mathcal{W}}|\rho(w)|$, where $|\rho(w)|$ is the minimal length of the path $\rho(w)$.  Given $\mathcal{W}$, define the \emph{$\mathcal{W}$-minimal poset} \[K_{\min(\mathcal{W})}=\bigcup_{\text{min } \| \cdot \|_\mathcal{W}} \st(R_n,\rho), \] where $\st(R_n,\rho)$ denotes the simplicial star of $(R_n,\rho)$ in $K_n$ and we take the union over all roses which minimize the $\mathcal{W}$-norm. In their original paper \cite{CV86} on outer space, Culler--Vogtmann prove
\begin{theorem} \label{ContractibleMinSet}For any $\mathcal{W}$, the $\mathcal{W}$-minimal poset $K_{\min(\mathcal{W})}$ is contractible. 
\end{theorem}
The minimal poset for a carefully chosen $\mathcal{W}$ will be the starting point of the proof of Theorem \ref{ColoredCollins}. We will also need the following homotopy theoretic result due to Quillen.

\begin{theorem}[Quillen \cite{qui73}]Let $P$ be a poset and let $f:P\rightarrow P$ be a poset map such that for all $x\in P$, we have $f(x)\leq x$. Then $f$ induces a deformation retraction of simplicial realizations $K(P)\rightarrow K(f(P))$.  
\end{theorem}

\textbf{Contracting onto colored roses.}
Fix a free splitting $F_n=F_{\ell_1}*\cdots*F_{\ell_k}$, and choose a basis $\{x_1,\ldots x_n\}$ for $F_n$ such that $\{x_{\ell_{(i)}+1},\ldots, x_{\ell_{(i+1)}}\}$ is a basis for $F_{\ell_{i+1}}$, where $\ell_0=0$. Let $\mathcal{W}_i$ for $1\leq i\leq k$ denote the set of conjugacy classes of every generator and every product of two distinct generators in $F_{\ell_i}$.  Set $\mathcal{W}=\cup_{i=1}^k\mathcal{W}_i$.  We have

\begin{proposition}\label{MinAut} A marked rose $(R_n,\phi)\in K_{\min(\mathcal{W})}$ if and only if $\phi\in \mathrm{P}\Sigma(\ell_1,\ldots, \ell_k)$ $($after possibly permuting/inverting basis elements$)$.   
\end{proposition}
\begin{proof} It is clear that any element of $\mathrm{P}\Sigma(\ell_1,\ldots, \ell_k)$ preserves $\mathcal{W}$.  For the other direction, first we observe that $\{x_1,\ldots, x_n\}\subset \mathcal{W}$, hence $\phi$ is a symmetric automorphism, up to a product of inversions of some $x_i$.  After precomposing with a permutation, we know that $\phi$ is pure symmetric, i.e. for each $i$ there exists $v_i\in F_n$ such that $\phi:x_i\mapsto v_ix_iv_i^{-1}$.  To prove the claim, we must show that for each pair of generators $x_i,x_j\in F_{\ell_r}$, we have $v_i=v_j$. Consider any pair $x_i,x_j\in F_{\ell_r}$. Then  \[\phi:x_ix_j\mapsto (v_ix_iv_i^{-1})(v_jx_jv_j^{-1}).\] Conjugate by $v_j^{-1}$ to rewrite $\phi(x_ix_j)$ as $ux_iu^{-1}x_j$, where $u=v_j^{-1}v_i$. Now we claim that either $v_i=v_j$, or the length of the conjugacy class of $\phi(x_ix_j)$ is at least 4, contradicting our assumption that $(R_n,\phi)$ has minimal norm. If this length is strictly less than 4 and $u\neq 1$, we must have either $u=x_i^{\pm1}$ or $u=x_j^{\pm1}$.  Without loss of generality, suppose $u=x_i$.  Then $v_i=v_jx_i$ and \[\phi(x_i)=(v_jx_i)x_i(x_i^{-1}v_j^{-1})=v_jx_iv_j^{-1},\] hence $\phi$ conjugates $x_i$ by $v_j$ as well. This proves that $\phi$ conjugates each free factor by the same element, hence $\phi\in \mathrm{P}\Sigma(\ell_1,\ldots,\ell_k)$. 
\end{proof}

It follows from the proposition that $K_{\ell_1,\ldots,\ell_k}=K_{\min(\mathcal{W})}$. By Theorem \ref{ContractibleMinSet}, we know that $K_{\min(\mathcal{W})}$ is contractible, which finishes the first step of the proof.  For the second step of the proof, we appeal to the deformation retraction Collins defined, onto the set of symmetric graphs. Let $(X,\rho)\in CV_n$ be a marked graph.  Collins proves that there exists a forest $E_{-\Sigma}(X)\subset X$ such that the following holds.

\begin{proposition} \label{AsymmetricEdge}Let $T\subset X$ be a maximal tree with collapse map $c:X\rightarrow X/T=R_n$.  Suppose the induced homotopy equivalence $c\circ \rho:R_n\rightarrow R_n$ is a symmetric automorphism, i.e. $[c\circ\rho]\in \mathrm{P}\Sigma_n$. Then $E_{-\Sigma}(X)\subset T$. Moreover, $\Sigma(X):=X/E_{-\Sigma}(X)$ is symmetric.
\end{proposition}

Now suppose $(X,\rho)\in K_{\min(\mathcal{W})}$.  Since $\mathrm{P}\Sigma(\ell_1,\ldots,\ell_k)\leq \mathrm{P}\Sigma_n$, it follows from Proposition \ref{AsymmetricEdge} that if $T\subset X$ is any maximal tree such that the collapse map induces an automorphism in $\mathrm{P}\Sigma(\ell_1,\ldots,\ell_k)$, then $E_{-\Sigma}(X)\subset T$. The collapse $(X,\rho)\mapsto (\Sigma(X), \Sigma(\rho))$ thus gives a well-defined $\mathrm{P}\Sigma(\ell_1,\ldots,\ell_k)$-equivariant map \[h_1:K_{\min(\mathcal{W})}\rightarrow \left ( K_{\min(\mathcal{W})} \right )^\Sigma,\] where $\left (K_{\min(\mathcal{W})} \right )^\Sigma$ is the sub-poset of all symmetric graphs in $K_{\min(\mathcal{W})}$. Because  $h_1(X,\rho)\leq(X,\rho)$ for all $(X,\rho)\in K_{\min(\mathcal{W})}$, and $h_1(X,\rho)$ is symmetric Quillen's theorem implies that $h_1$ is a homotopy equivalence. This completes the second step of the proof. 

For the final step, we will prove an analogous version of Proposition \ref{AsymmetricEdge} for colored symmetric graphs. For any symmetric graph, any two edges lie on a unique embedded circle.  If $(X,\rho)\in \left ( K_{\min(\mathcal{W})} \right )^\Sigma$ and $T\subset X$ is a maximal tree such that the collapse map induces an automorphism in $\mathrm{P}\Sigma(\ell_1,\ldots, \ell_k)$, we say $T$ is \emph{$\mathcal{W}$-compatible}. By definition, for any $(X,\rho)\in  \left ( K_{\min(\mathcal{W})} \right )^\Sigma$ there is at least one $\mathcal{W}$-compatible tree $T_0$.  From collapsing $T_0$ to form a colored rose, we get an induced coloring of cycles on $X$.  Now consider any maximal tree $T\subset X$.  For every edge $e\in T$, there is a unique \emph{companion} edge $e'\in X\setminus T$ such that $e'$ connects the two components of $T\setminus \{e\}$ \cite[Lemma 4.2]{Col89}.  

An edge $e\in T$ is called \emph{$\mathcal{W}$-incoherent rel $T$} if $e\cup e'$ separates two cycles of the same color. We define \[E_{-\mathcal{W}}(T)=\{e\in T \,| \, \mbox{$e$ is $\mathcal{W}$-incoherent rel $T$}\}.\]

Note that for any cycle $\gamma$ of length $m$ in $X$, a maximal tree contains exactly $m-1$ of the edges in $\gamma$. The extension of Proposition \ref{AsymmetricEdge} to our setting is that $E_{-\mathcal{W}}(T)$ is independent of $T$. 

\begin{proposition} \label{WIncoherent} Let $(X,\rho)\in \left ( K_{\min(\mathcal{W})} \right )^\Sigma$ and suppose $T,T'\subset X$ are maximal, $\mathcal{W}$-compatible trees.  Then \[E_{-\mathcal{W}}(T)=E_{-\mathcal{W}}(T').\]
\end{proposition}
\begin{proof} Since $X$ is symmetric and $T$ is $\mathcal{W}$-compatible, we can color every cycle of $X$  using $T$ and consider each cycle as labelled by a specific conjugate of the generators $v_ix_iv_i^{-1}$, where if $x_i$ and $x_j$ belong to the same $F_{\ell_r}$, then $v_i=v_j$. Now suppose $e\in T$ is $\mathcal{W}$-incoherent.  Then $e$ and its companion edge $e'$ lie on some cycle $\gamma_0$ labelled by $vxv^{-1}$ for some generator $x$, say.  By assumption$e\cup e'$ separate two loops $\gamma_1$ and $\gamma_2$ of the same color labelled by $uyu^{-1}$ and $uzu^{-1}$, respectively, for some generators $y$, and $z$. 

Suppose for contradiction that $e\notin T'$.  Then $e'\in T'$ and if we collapse $T'$ we can consider the marking change from $X/T\rightarrow X/T'$.  Let $X_1\cup X_2$ be the two components of $X\setminus\{e,e'\}$, so that $\gamma_i\subset X_i$, $i=1,2$. Suppose that $y= y_1,\ldots y_s$ are the generators whose conjugates label cycles in $X_1$, and $z=z_1,\ldots, z_{n-s+1}$ are the generators whose conjugates label cycles in $X_2$.   Choosing a basepoint in $X_1$ on $\gamma_0$, the change of marking $X/T\rightarrow X\rightarrow X/T'$ conjugates every cycle in $X_2$ by $vxv^{-1}$, but none from $X_1$.  Explicitly, under the change of marking we have \[uyu^{-1}\mapsto w_1uyu^{-1}w_1^{-1},\] where $w_1$ is a product of conjugates the $y_i$, and \[uzu^{-1}\mapsto (vx^{\pm1}v^{-1})w_2(uzu^{-1})w_2^{-1}(vx^{\mp1}v^{-1}),\] where $w_2$ is a product of conjugates of the $z_j$. Note that since the change of marking is induced from a change of maximal tree, at most one conjugate of $y_i^{\pm1}$ occurs in $w_1$ and at most one conjugate of $z_j^{\pm1}$ occurs in $w_2$.  Comparing the images on $y$ and $z$, if $T'$ is $\mathcal{W}$-compatible, we must have that 
\begin{equation}\label{ConjCoefficient}w_1u=vx^{\pm1}v^{-1}w_2u\Rightarrow w_1=vx^{\pm 1}v^{-1}w_2.\end{equation}
Looking at their images in the abelianization, we obtain $[w_1]=\sum_{i=1}^sa_i[y_i]$ while the right hand side is $\pm[x]+\sum_{j=1}^{n-s+1}b_j[z_j]$, where $a_i, b_j\in\{-1,0,1\}.$ Thus, the left- and right-hand sides of equation (\ref{ConjCoefficient}) cannot possibly be equal, and we conclude that $T'$ is not $\mathcal{W}$-compatible, contradicting our hypothesis.
%
\end{proof}

\begin{proof}[Proof of Theorem~\ref{ColoredCollins}]It follows from Proposition \ref{WIncoherent} that we can define $E_{-\mathcal{W}}(X)$ to be $E_{-\mathcal{W}}(T)$ for any $\mathcal{W}$-compatible tree in $X$. In particular, $E_{-\mathcal{W}}(X)\subset T$ is a forest. From the definition of $\mathcal{W}$-incoherent, we see that $X/E_{-\mathcal{W}}(X)$ a is colored symmetric graph. The collapse $X\mapsto X/E_{-\mathcal{W}}(X)$ therefore gives a well-defined $\mathrm{P}\Sigma(\ell_1,\ldots,\ell_k)$-equivariant map
\[h_2: \left ( K_{\min(\mathcal{W})} \right )^\Sigma\rightarrow \mathrm{C}\Sigma(R_{\ell_1},\ldots, R_{\ell_k}). \]
 Since $h_1(X,\rho)$ is colored symmetric for all $(X\rho)\in \left ( K_{\min(\mathcal{W})} \right )^\Sigma$, we can apply Quillen's theorem to see that $h_2$ is also a homotopy equivalence.  In particular, $\mathrm{C}\Sigma(R_{\ell_1},\ldots, R_{\ell_k})$ is contractible.  Since $\mathrm{P}\Sigma(\ell_1,\ldots,\ell_k)$ is torsion-free, its action on $K_{\min(\mathcal{W})}\subset K_n$ is free, and by Proposition \ref{MinAut} the quotient is compact.  Equivariance of $h_1$ and $h_2$, and hence $h=h_2\circ h_1$, gives a homotopy equivalence 
 \[\overl{h}: K_{\min(\mathcal{W})}/\mathrm{P}\Sigma(\ell_1,\ldots,\ell_k)\rightarrow \mathrm{C}\Sigma(R_{\ell_1},\ldots, R_{\ell_k})/\mathrm{P}\Sigma(\ell_1,\ldots,\ell_k).\] 

In order to finish the proof, we just need to identify $\mathrm{C}\Sigma(R_{\ell_1},\ldots, R_{\ell_k})/\mathrm{P}\Sigma(\ell_1,\ldots,\ell_k)$ with $W(R_{\ell_1},\ldots,R_{\ell_k})$. To see this, fix once and for all a metric and a coloring on $R_n$. Using the fact that colored symmetric graphs are \emph{a fortiori} symmetric, we consider all marked, colored symmetric blow-ups $X$ where the length of each colored cycle is the same as in the the colored rose $R_n$.  This gives a $\mathrm{P}\Sigma(\ell_1,\ldots,\ell_k)$-equivariant embedding $\psi:\mathrm{C}\Sigma(R_{\ell_1},\ldots, R_{\ell_k})\hookrightarrow CV_n$.  The quotient of $\text{im}(\psi)$ by $\mathrm{P}\Sigma(\ell_1,\ldots,\ell_k)$ is $W(R_{\ell_1},\ldots,R_{\ell_k})$. This proves the theorem.\end{proof}


\section{The Fibers of the Period Map on Torelli Space}\label{PeriodFiber}
In this section we will combine results of the previous sections to describe the fibers of the period mapping $\Phi:\calT_n\rightarrow\calQ_n$. If we cast the results of Section \ref{Decomp} in terms of the spaces defined in Section \ref{Wedge} we obtain the following theorem, which suffices to prove Theorem A.

\begin{theorem}\label{Fiber} Let $M$ be a positive definite quadratic form in the image of $\Phi$, and let $D$ be a connected component of $\Phi^{-1}(M)$. Then
\begin{enumerate}
\item $D$ admits a quasi-fibration over a product of simplices where the fibers are all configuration spaces of wedges of graphs.
\item The stabilizer of $D$ in $GL_n(\Z)$ is isomorphic to $\Isom(\mathcal{C}(\Gamma))$ for some maximally separated $\Gamma$.     
\item $D$ is aspherical and its fundamental group injects into $\pi_1(\calT_n)$. Moreover, the image of $\pi_1(D)$ is a conjugate of $\mathrm{P} \Sigma (\ell_1, \dots, \ell_k)$ for some choice of $\ell_i$ ($1 \leq i \leq k$). 
\end{enumerate}
\end{theorem}
Before proceeding to the proof, we record the following.

\begin{remark} Given a positive definite quadratic form $M$ in the image of $\Phi$, parts (1) and (2) of the theorem allow us to completely describe the fiber $\Phi^{-1}(M)$ as well as the action of $\text{stab}_{\mathrm{GL}_n(\Z)}(M)$ on the components of the fiber.  
\end{remark} 

\begin{remark}Bridson--Vogtmann \cite{BV95} showed that the spine of $CV_n$ ($n \geq 3$) does not support a CAT(0) metric. The presence of CAT(0) subspaces in $CV_n$, on the other hand, is open in general. It may be shown that when $\mathcal{C}(\Gamma)$ has at most three components, the fiber containing $(\Gamma, \rho)$ is locally CAT(0), and hence lifts to a CAT(0) subspace of $CV_n$, when endowed with the simplicial metric. When $\mathcal{C}(\Gamma)$ has four or more components, the simplicial metric on the fiber is not locally CAT(0).
\end{remark}
%

The proof will follow from a sequence of lemmas.  Fix some form $M$ in the image of $\Phi$, and consider the fiber $\Phi^{-1}(M)$.  In general, the fiber may be disconnected.  Choose a point $(\Gamma,\rho)\in \Phi^{-1}(M)$ and consider the connected component $D_{(\Gamma,\rho)}$ containing $(\Gamma,\rho)$. Let $\mathcal{C}(\Gamma)=\coprod_{i=1}^k\Gamma_i$ denote the splitting graph of $\Gamma$.  After acting by $\GL_n(\Z)$ we can assume that the marking on $\Gamma$ decomposes as a product $(\Gamma,\rho)=\bigoplus_{i=1}^k (\Gamma_i,\rho_i)$, with a corresponding decomposition $M=\bigoplus_{i=1}^kM_i$.    Proposition \ref{splittingblow-up} and Lemma \ref{uniqueblow-up} imply that $D_{(\Gamma,\rho)}$ consists of all the marked graphs that can be obtained from $\mathcal{C}(\G)$ by the following two moves:
\begin{enumerate}
\item Expand or collapse a pair of separating edges in the $\G_i$.
\item Wedge together the $\G_i$.  
\end{enumerate}
These two operations interact with each other in a compatible and essentially disjoint way, as the following lemma shows.
\begin{lemma}\label{quasifibration} $D_{(\Gamma,\rho)}$ has the structure of a complex of spaces with base space a product of simplices and the fiber a configuration space of wedges. In particular, $D_{(\Gamma,\rho)}$ admits a quasi-fibration over a product of simplices.   
\end{lemma}
\begin{proof} Suppose $\Gamma_i$ has $k_i$ equivalence classes of separating edges of size $r_{i,1},\ldots,r_{i,k_i}$.  Denote by $\Delta^r$ the standard $r$-simplex. We define a map 
\[ l=\prod_{i=1}^kl_i: D_{(\Gamma,\rho)}\rightarrow\Delta_{(\Gamma,\rho)}= \prod_{i=1}^k\prod_{j=1}^{j=k_i}\Delta^{r_{i,j}}, \]
 where $l_i$ records the point in each simplex corresponding to the relative lengths of edges in each equivalence class of separating edges in $\Gamma_i$. Proposition \ref{splittingblow-up} implies that this is well-defined no matter how the components $\Gamma_i$ are wedged together. 
 
Now fix a point $p$ in the target $\Delta_{(\Gamma,\rho)}$, and a point in $l^{-1}(p)$, which by a slight abuse of notation, we also call $(\G,\rho)$. We claim that $l^{-1}(p)$ is the configuration space of all possible ways of wedging together the components of $\mathcal{C}(\G)$. By Proposition \ref{SepPair}, none of the lengths of edges in $\mathcal{C}(\G)$ change within a fiber, hence there is a surjective map \[\tilde{w}:\widetilde{W}(\G_1,\ldots,\G_k)\rightarrow l^{-1}(p).\]  Suppose $\tilde{w}$ is not injective.  Then there exist $x_1, x_2\in \widetilde{W}(\G_1,\ldots,G_k)$ with $\tilde{w}(x_1)=\tilde{w}(x_2)$. There is a fixed a homology-marking on each of the $\Gamma_i$, and since none of the $\Gamma_i$ have trivial homology, this implies that $x_1$ and $x_2$ have to come from the same pattern $\widetilde{T}(\G_1,\ldots,\G_k)\subset \widetilde{W}(\G_1,\ldots,G_k)$.  By definition, there must be an isometry from the graph determined by $x_1$ and the graph determined by $x_2$, which preserves each component of $\mathcal{C}(\Gamma)$.  Thus there are isometries $f_i:\Gamma_i\rightarrow \Gamma_i$, which map the wedge points on $\Gamma_i$ determined by $x_1$ to the wedge points determined by $x_2$, and which act trivially on the homology-marking. But if $\G_i\ncong S^1$ and $f_i$ acts trivially on homology, then $f_i$ must be the identity.  If $\G_i\cong S^1$ and $f_i$ acts trivially on homology then $f_i$ is a rotation.  By definition of $W(\G_1,\ldots,\G_k)$, we have that $\tilde{w}$ factors through a map $w:W(\G_1,\ldots,\G_k)\rightarrow l^{-1}(p)$, which is a homeomorphism.  

On the interior of each product of faces of $\Delta_{(\Gamma,\rho)}$, the homeomorphism type of each fiber is the same.  Forest collapses induce maps between configurations spaces of wedges, which are homotopy equivalences by Proposition \ref{Config}. Thus $D_{(\Gamma,\rho)}$ has the structure of  a complex of spaces where all face maps are homotopy equivalences.  Since the base is contractible, the fact that all fibers are homotopy equivalent implies that $l:D_{(\Gamma,\rho)}\rightarrow\Delta_{(\Gamma,\rho)}$ is a quasi-fibration.  
\end{proof}


To prove the second part of Theorem~\ref{Fiber}, we need to analyze the stabilizer of a marked graph in the fiber.  By the $\GL_n(\Z)$-equivariance of $\Phi$ and the fact that $\mathcal{I}_n$ is torsion-free, we know that the stabilizer of $(\Gamma,\rho)\in \Phi^{-1}(M)$ injects into $\text{stab}_{\GL_n(\Z)}(M)$. First, we need a lemma about maximal blow-ups.  

\begin{lemma} \label{BaryIsom} Suppose $\Gamma$ does not have separating vertices. Then there exists a maximal blow-up $\Gamma'$ of $\Gamma$ such that $\Isom(\Gamma)\leq \Isom(\Gamma')$.
\end{lemma}
\begin{proof} By Lemma \ref{uniqueblow-up}, there is a unique way to obtain a maximally separating blow-up $\Gamma'$ of $\Gamma$.  Let $\Gamma'$ be the maximal blow-up of $\Gamma$ in which all the edges in a given pairwise separating edge class have the same length. We claim that any isometry $f:\Gamma\rightarrow \Gamma$ extends to an isometry $f':\Gamma'\rightarrow\Gamma'$.  First observe that $f$ preserves the following relations:\begin{enumerate}
\item If $e\cup e'$ separate then $f(e)\cup f(e')$ separate, and 
\item If $e\cup v$ separate then $f(e)\cup f(v)$ separate.  
\end{enumerate}
In a manner similar to the proof of Lemma \ref{Equivalence}, we can also define equivalence classes of separating edge-vertex pairs.   Since $f$ takes vertices to vertices and edges to edges hence we conclude that $f$ permutes equivalence classes of pairwise separating edge and separating edge-vertex pairs. Second, note that if $[e]$ denotes one pairwise separating edge class, and $[e']=f([e])$, then the sum of the lengths of all edges in $[e]$ is the same as the sum of the lengths of all edges in $[e']$.

Let $v_1,\ldots, v_r$ be the vertices belonging to a separating edge-vertex pair and let $e_1,\ldots, e_s$ be edges belonging to a separating edge or edge-vertex pair. Note that a single vertex may separate $\Gamma$ with more than one edge class.  Then $f$ permutes the components of $\Delta=(\Gamma\setminus (e_1\cup\cdots\cup e_s))||(v_1\cup\cdots\cup v_r)$.  The lengths of edges in $\Delta$ do not change when passing from $\Gamma$ to $\Gamma'$.  Now if $f(e)=e'$ then $l(e)=l(e')$ and if $f(v)=v'$ we must have that $v$ and $v'$ belong to the same number of separating edge-vertex pair classes.  Therefore the local topology at $v$ and $v'$ is the same, and after we  blow-up the unique edges in each separating edge-vertex pair given by Lemma \ref{uniqueblow-up}, the topology remains the same.  We define $f'$ to be the same combinatorial map on $\Gamma'\setminus\{\overl{e_1}',\ldots,\overl{e_p}'\}$ and there is a unique way to extend $f'$ over $\overl{e_1}',\ldots, \overl{e_p}'$.  The condition on lengths of edges in a pairwise separating edge class ensures $f'$ will be an isometry.    
\end{proof}

We call the maximal graph in which all separating edge classes have constant length the \emph{barycenter}. If $\Gamma$ has separating vertices, and $\mathcal{C}(\Gamma)=\coprod_{i=1}^n\Gamma_i$, then $\Gamma$ is a barycenter if $\Gamma_i$ is a barycenter for each $i$.  If $\Gamma$ is a barycenter without separating vertices, then $\G$ is the only graph in $D_{(\Gamma,\rho)}$ with this property.  On the other hand if $\G$ has separating vertices, $\G$ is no longer the unique barycenter in $D_{(\Gamma,\rho)}$ but any two barycenters $\G_1, \G_2\in D_{(\Gamma,\rho)}$ satisfy $\mathcal{C}(\G_1)\cong \mathcal{C}(\G_2)$. The following lemma says the isometry group of $\mathcal{C}(\Gamma)$ determines the stabilizer of the component $D_{(\Gamma,\rho)}$. For ease of notation, if $M$ is a $k\times k$ positive definite symmetric matrix, instead of writing $\text{stab}_{\GL_k(\Z)}(M)$ we will write $\Aut(M)$.

\begin{lemma}\label{SepComponent} If $(\G,\rho)\in\Phi^{-1}(M)$ is a barycenter, the stabilizer of the component $D_{(\Gamma,\rho)}$ is $\Isom(\mathcal{C}(\Gamma))$.
\end{lemma}
\begin{proof} 

Recall that $M=M_1\oplus\cdots\oplus M_k$ is block diagonal, and $M_i$ is just the image period mapping for $\G_i$.  We have a decomposition $\Aut(M)=\left(\Aut(M_1)\times\cdots\times \Aut(M_k)\right)\rtimes P$, where $P$ permutes factors $M_i$ which are isomorphic. If $\sigma\in P$ is a permutation that sends $M_i$ to $M_{\sigma(i)}$, then $\sigma$ fixes the component containing $\G$ if and only if $\G_i$ is isomorphic to $\G_{\sigma(i)}$ for each $i$. For if $\G_i\cong\G_{\sigma(i)}$ for each $i$, then there exists a point in $D_{(\Gamma,\rho)}$ where $\G_i$ and $\G_{\sigma(i)}$ are wedged together at the same point, and there is therefore an isometry which realizes $\sigma$.  Conversely, if $\G_i\ncong \G_j$, then since $\G_i$ and $\G_j$ are maximal and do not have separating vertices, Lemma \ref{BaryIsom} implies that no separating edge collapse of $\G_i$ is isomorphic to a separating edge collapse of $\G_j$. Because $\tau$ takes isometric graphs to isometric graphs, it is not possible for $\tau$ to take any point of $D_{(\Gamma,\rho)}$ to another point of $D_{(\Gamma,\rho)}$.  We conclude that either $\tau \in P$ permutes the connected components of $\Phi^{-1}(M)$ or has a fixed point in $D_{(\Gamma,\rho)}$, and hence comes from an isometry of $\mathcal{C}(\Gamma)$. 

Now suppose $f\in \Aut(M_1)\times\cdots\times \Aut(M_k)$. In this case, $f$ acts on $\mathcal{C}(\Gamma)$ by preserving each of the components $\G_i$.  Let $f_i=f_{|_{\G_i}}$. Since each $\G_i$ is a maximal and has no separating vertices, either $f_i$ is an isometry of $\G_i$ or changes the marking. In the latter case, $f$ does not preserve the component $D_{(\Gamma,\rho)}$. In the former, $f$ is an isometry of $\mathcal{C}(\Gamma)$.
%
\end{proof}

Finally, we prove the third statement of Theorem~\ref{Fiber}. Lemma~\ref{quasifibration} states that the $D_{(\G,\rho)}$ is homotopy equivalent to $W(\G_1,\ldots,\G_k)$, where $(\G,\rho)$ is a barycenter, and $\mathcal{C}(\G)=\coprod_{i=1}^k\G_k$. By Theorem~\ref{ColoredCollins}, $W(\G_1,\ldots,\G_k)$ is homotopy equivalent to $W(R_{\ell_1},\ldots,R_{\ell_k})$ where $l_i$ is the rank of $\G_i$. The latter is aspherical and has fundamental group isomorphic to $\mathrm{P} \Sigma (\ell_1, \dots, \ell_k)$. 

To see that the inclusion $j:W(\G_1,\ldots,\G_k)\hookrightarrow\calT_n$ is $\pi_1$-injective, we observe that by choosing a maximal tree in each $\G_i$, the homotopy equivalence between $W(\G_1,\ldots,\G_k)$ and $W(R_{\ell_1},\ldots,R_{\ell_k})$ actually takes place in $\calT_n$. Lifting to the universal cover, we get a $\mathrm{P} \Sigma (\ell_1, \dots, \ell_k)$-equivariant homotopy equivalence $\tilde{j}$ between the universal cover of $W(\G_1,\ldots,\G_k)$ and some lift of $\mathrm{C}\Sigma(R_{\ell_1},\ldots, R_{\ell_k})$. Choosing some marked rose in the image of $\tilde{j}$, we obtain that $j_*\pi_1(W(\G_1,\ldots,\G_k))$ is actually a conjugate of $\mathrm{P} \Sigma (\ell_1, \dots, \ell_k)$.

 
\subsection{Discrete and indiscrete fibers}\label{Examples}
We now give some examples of fibers of the period mapping in order to demonstrate some of the variety of behavior that can occur.  Recall that a graph $\Gamma$ is called \emph{hyperelliptic} if it admits an isometric involution $\iota:\Gamma\rightarrow \Gamma$ such that $\iota_*$ acts as $-\text{Id}$ on $H_1(\G)$.  The different properties we will consider are discrete and indiscrete fibers, trivalent and non-trivalent graphs, and hyperelliptic vs. non-hyperelliptic graphs.  

For this section we first introduce some terminology.  A graph is \emph{3-connected} if any pair of vertices does not disconnect the graph, and \emph{3-edge connected} if no pair of edges separates the graph.   Let $v_1,v_2$ be a separating pair of vertices in a graph $\Gamma$.  A \emph{Whitney 2-move} at $v_1,v_2$ is defined as follows.  Let $\Delta$ be a component of $\Gamma\setminus\{v_1,v_2\}$. Form a new graph $\Gamma'$ from $\Gamma$ by taking all of the edges of $\Delta$ incident at $v_1$ and attaching them to $v_2$, and vice versa.  $\Gamma'$ is said to be obtained from $\Gamma$ by a Whitney 2-move.  

\begin{figure}[h]
\centering
\includegraphics[width=1in]{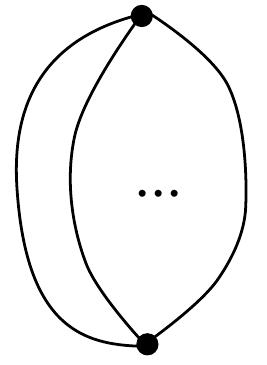}
\caption{A generalized theta graph}
\label{ThetaEx1}
\end{figure}

\begin{example}[Single point fiber]   The generalized theta graph (rank $\geq2$) pictured in figure \ref{ThetaEx1} is an example where the fiber has exactly one point for any choice of lengths of the edges. It is hyperelliptic, but if rank $ \geq3$ it is clearly not trivalent.  
\end{example}

\begin{figure}[h]
\centering
\includegraphics[width=3in]{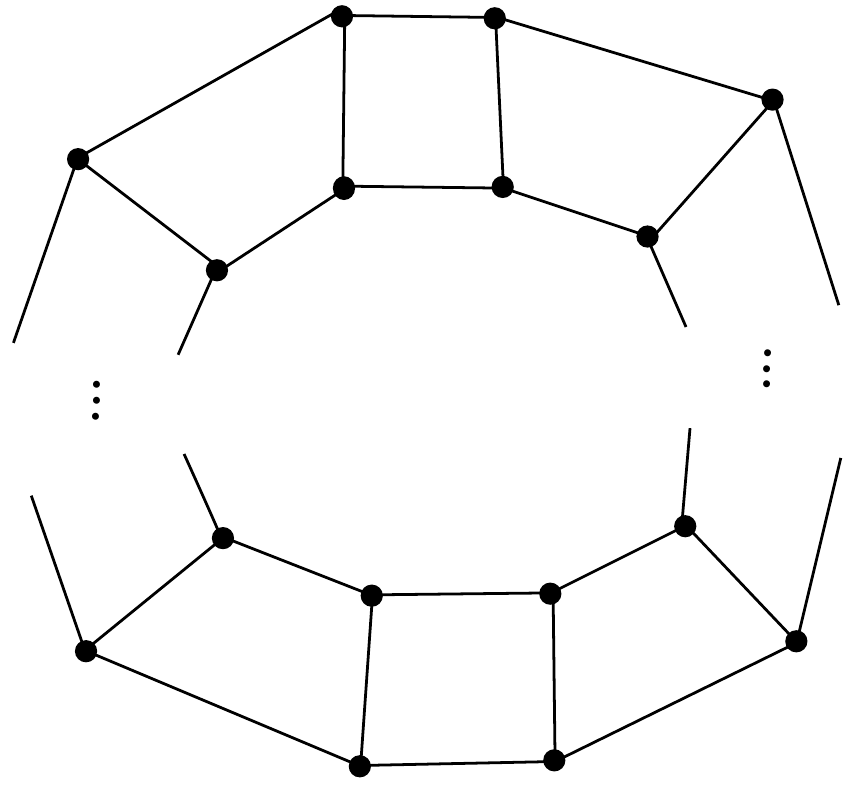}
\caption{A chain of squares.}
\label{CycleEx2}
\end{figure}

\begin{example}[Two point fiber] A chain of squares arranged in a circle (rank $\geq 4$) is trivalent and has no non-trivial Whitney 2-moves. Hence, by the theorem of Caporaso-Viviani \cite{CaVi11}, there is no other isomorphism type of graph in the same fiber.  Since there are no separating vertices or separating pairs of edges, Theorem \ref{Fiber} implies that each connected component of the fiber is a single point.  If all of the squares are congruent, every automorphism of the intersection pairing is induced by a graph automorphism, except multiplication by $\{\pm I\}$. Since the graph is not hyperelliptic, the fiber consists of exactly two points, which are exchanged by $\pm I\in \GL_n(\Z)$. 
\end{example}

\begin{figure*}[h]
\centering
\begin{subfigure}[t]{0.5\textwidth}
\centering
\includegraphics[width=1.75in]{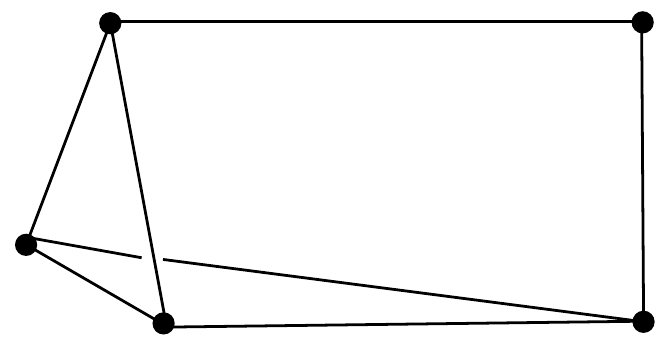}
\caption{The graph $A_1$.}
\label{A1Ex3}
\end{subfigure}
\quad
\begin{subfigure}[t]{0.5\textwidth}
\centering
\includegraphics[width=3in]{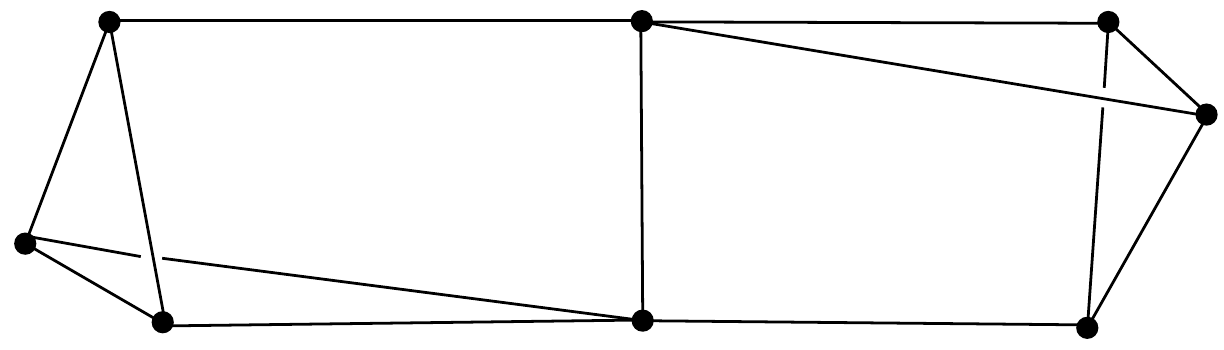}
\caption{A graph in $A_2$}
\label{A2Ex3}
\end{subfigure}
\caption{Examples of graphs in the families $A_n$.}
\end{figure*}

\begin{example}[Discrete fiber, different isomorphism types] The following families $A_n$ give graphs of rank $3n$ where fibers are discrete since there are no separating vertices or pairs of edges.  An element of $A_n$ is constructed by taking $n$ copies of $A_1$ and identifying them along the edge $e_0$.  The only non-trivial Whitney 2-moves come from inverting some number of $A_1$-components along the shared edge.  They are not-hyperelliptic, so there are at least $2^{n+1}$ homology-marked graphs in the fiber of $\Phi$. When we pass to $\mathcal{G}_{3n}$, the isomorphism type is entirely determined by how many $A_1$ pieces are up or down, thus there are exactly $n+1$ points in the fiber of $\overl{\Phi}$.
\end{example}

Since we assume our graphs do not have separating edges, in $\calT_n$ \emph{most} graphs do not have separating vertices, in the sense that the open dense subset of trivalent graphs do not have separating vertices.  It follows from Theorem \ref{Fiber} that for these graphs the fiber is particularly simple: it is a union of contractible simplicial complexes.  Our final family of examples is of this type.

\begin{figure}[h]
\centering
\includegraphics[width=2in]{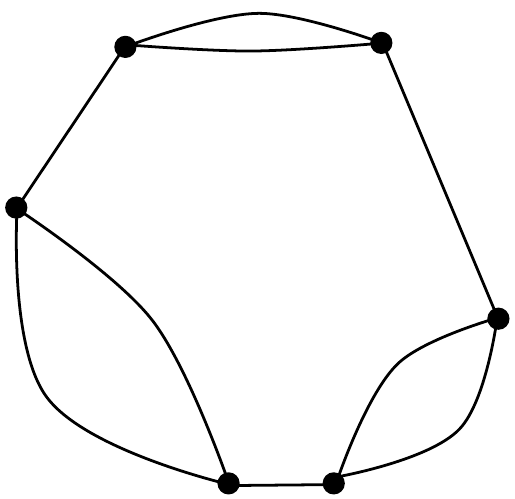}
\caption{Doubled-edge polygon with $n=3$.}
\label{PolygonEx4}
\end{figure}

\begin{example}[Indiscrete, contractible fibers] Consider a polygon with $2n$ sides, and replace every other edge by a double edge.  This graph is trivalent, has rank $n+1$ and a single separating edge class of size $n$.  Any pair of vertices separate, so there are always non-trivial Whitney 2-moves, and none of these are induced by automorphisms.  The number of components is therefore determined by the number of Whitney 2-moves and the lengths of the double edges.  In the perfectly symmetric case, the number of components is $(n-1)!$, each of which is an $n$-simplex.  Note that by gluing these polygons together one obtains components which are more products of simplices.  
\end{example}

We close this subsection with a observation about the hyperelliptic locus. One might think it possible for some fiber to contain some components which intersect the hyperelliptic locus and some of which do not.  The following proposition says this is not the case.  

\begin{proposition}If some connected component of $\Phi^{-1}(M)$ contains a hyperelliptic graph, then every component does.  
\end{proposition}
\begin{proof} Suppose $(G_1,\rho_1),(G_2,\rho_2)\in\Phi^{-1}(M)$ where $G_2$ is hyperelliptic and $G_2$ is not.  If $(G_1,\rho_1)$ and $(G_2,\rho_2)$ lie in the same connected component, we are done. Otherwise, by Caporaso-Viviani and Theorem \ref{Fiber}, we know that $G_1$ and $G_2$ are related by Whitney 2-moves and separating edge collapses. Observe that Whitney 2-moves preserve separating edge pairs, so it suffices to show that any Whitney 2-move applied to a hyperelliptic graph is hyperelliptic.  

To see this, suppose $G$ is hyperelliptic with involution $\iota:G\rightarrow G$.  Then $G/\iota=T$ is a tree with marked vertices.  $G$ is obtained from $T$ by doubling $T$ along the marked vertices.  We denote by $l$ the axis of the involution $\iota$ on $G$. Suppose $v_1$ and $v_2$ are the two vertices of $G$ along which a Whitney 2-move is applied.  Note that neither $v_1$ nor $v_2$ is a leaf of $T$.  We divide the proof into various cases.

\textbf{Case 1:} $v_1$ and $v_2$ are both marked. Each component of $G\setminus\{v_1,v_2\}$ is hyperelliptic.  It follows that a Whitney 2-move along these vertices does nothing to $G$.  

\textbf{Case 2:} $v_1$ is marked and $v_2$ is unmarked. Without loss of generality, $v_1$ and $v_2$ lie in $T$.  Let $\alpha$ denote the path in $T$ between $v_1$ and $v_2$.  If there exists a vertex on $\alpha$ other than $v_1, v_2$ then $v_1\cup v_2$ does not separate $G$.  Else, $\alpha$ is a single edge and the Whitney 2-move does nothing.  

\textbf{Case 3:} $v_1$ and $v_2$ both unmarked.
\begin{enumerate}[(i)]
\item $v_1=\iota(v_2)$.  Every component of $G\setminus\{v_1,v_2\}$ is hyperelliptic, as in Case 1. 
\item $v_1, v_2\in T$ or $v_1, v_2\in \iota(T)$.  This is analogous to Case 2 above.  
\item $v_1\in T$, $v_2\in \iota(T)$, and $v_1\neq\iota(v_2)$.  Let $v_2'=\iota(v_2)$ and denote by $\alpha$ the path in $T$ between $v_1$ and $v_2'$.  Since $v_1\cup v_2$ separate, as above we conclude $\alpha$ is a single edge $e_0$.  Since $v_1$ and $v_2$ are each connected to some marked vertex, $G\setminus\{v_1,v_2\}$ consists of exactly 2 components. The Whitney 2-move along these vertices connects and endpoint of $e_0$ with an endpoint of $\iota(e_0)$, creating a single edge $E$ whose length is $l(e_0)+l(\iota(e_0))=2l(e_0)$.  Opposite $E$ is a vertex $v$, such that $v\cup E$ separates.  Blowing up an edge at $v$ creates a separating edge of length $l(e_0)$ gives hyperelliptic graph isometric to $G$.  
\end{enumerate}
This exhausts all cases and concludes the proof of the proposition.
\end{proof}


\section{Hyperelliptic automorphisms of free groups}

In this section, we define the hyperelliptic automorphism group $\HOut$ of a free group. We then find a finite generating set $\mathcal{S}$ for this group. We proceed, in Section~\ref{secTor}, to use $\mathcal{S}$ to obtain the generating set for the intersection $\STn$ of $\HOut$ with the Torelli subgroup of $\Outfn$ that is asserted by Theorem~\ref{hyptor}.

\subsection{The hyperelliptic involution of a free group}

Recall that, in analogy with mapping class groups, we define a \emph{hyperelliptic involution} of $F_n$ to be a (perhaps outer) automorphism of $F_n$ that has order 2 and induces $-I$ on $H_1(F_n, \Z) \cong \Z^n$. One example of such an involution is the automorphism we denote by $\iota \in \Autfn$, which inverts each member of the free basis $X$ of $F_n$.

Glover--Jensen \cite[Proposition 2.4]{GJ00} showed that $\Autfn$ has a unique hyperelliptic involution up to conjugacy. The same is true of such involutions in $\Outfn$, as we show in Lemma~\ref{hypuniq}, by proving that passing from $\Autfn$ to $\Outfn$ introduces no new hyperelliptic involutions. We thus refer to $\iota$ and $[\iota]$ as `the' hyperelliptic involutions of $F_n$, as they are essentially unique. 

\begin{lemma}\label{hypuniq}Any two hyperelliptic involutions in $\Outfn$ are conjugate.
\end{lemma}
\begin{proof}Let $[\rho] \in \Outfn$. We will show that $[\rho]$ is conjugate to the outer automorphism class of $\iota$.

In the sense of Culler \cite{Cul84}, we may realize $[\rho]$ as an order 2 automorphism of a finite, connected graph $\Delta$. Since any involution of a connected graph fixes either a vertex or an (undirected) edge, we may take as a basepoint of $\Delta$ a fixed point of this involution (subdividing a fixed edge, if necessary). Thus, the realization of $[\rho]$ provides a lift to an involution in $\Autfn$.

This lift of $[\rho]$ is conjugate to $\iota$, by Glover--Jensen's uniqueness theorem, with the conjugacy descending to $\Outfn$.
\end{proof}

Pushing the mapping class group analogy further, we now turn our attention to the automorphisms of $F_n$ that centralize $\iota$ or $[\iota]$, to obtain free group versions of hyperelliptic mapping classes.

\subsection{Hyperelliptic and palindromic automorphisms of free groups}

We define the subgroup of $\Outfn$ that centralizes $[\iota]$ to be the \emph{hyperelliptic automorphism group of $F_n$}, and denote it by $\HOut$. The remainder of this section is dedicated to obtaining a finite generating set for this group. As noted in Section 1, the group $\HOut$ has non-trivial intersection with the Torelli subgroup $\mathcal{T}_n$ of $\Outfn$. We denote this intersection by $\STn$, and refer to it as the \emph{hyperelliptic Torelli group}.

A group related to $\HOut$ is the \emph{palindromic automorphism group} $\pia_n \leq \Autfn$, which is defined to be the centralizer of the involution $\iota$. The group $\pia_n$ is so named because an automorphism $\alpha \in \pia_n$ must carry each $x \in X$ to a \emph{palindrome}; that is, $\alpha(x)$ must equal $\alpha(x)^{\mathrm{rev}}$, the result of writing the reduced word $\alpha(x)$ on $X$ in reverse.  It is elementary to show that $\pia_n$ contains no inner automorphisms, and so $\pia_n$ injects into $\HOut$ under the usual map $\Autfn \to \Outfn$. It is, however, a proper inclusion: note that the automorphism $\tau \in \HOut$ mapping each $x_i \in X$ to $x_1 x_i$ and fixing $x_1$ is not in its image. 

A finite generating set for $\pia_n$ was found by Collins \cite{Col95}, consisting of all permutations and inversions of the letter set $X$, and \emph{elementary palindromic automorphisms} $P_{ij}$ ($1 \leq i \neq j \leq n$), each of which maps $x_i$ to $x_j x_i x_j$ and fixes all other letters in $X$. Notationally, we shall denote by $\sigma_i$ the 2-cycle swapping $x_i$ and $x_{i+1}$, and by $\iota_j$ the inversion mapping $x_j$ to ${x_j}^{-1}$ that fixes the other letters of $X$. As we shall prove in Theorem~\ref{hgens}, in order to generate $\HOut$, it suffices to add the single element $[\tau]$ to the image of this set in $\HOut$. 

One strong motivation in particular for studying the group $\pia_n$ is that it virtually surjects onto the \emph{principal level 2 congruence subgroup} of $\GLn$, which we denote by $\Gamma_n[2]$, under the usual map $\Autfn \to \GLn$. The group $\Gamma_n[2]$ is the kernel of the mod 2 matrix entry reduction map $\GLn \to \GLnt$, and is a first example of the so-called \emph{congruence subgroups} of $\GLn$, which control the finite-index subgroup structure of $\GLn$ when $n \geq 3$ (cf. \cite{BLS64} and \cite{Men65}).

The (virtual) surjection $\pia_n \to \Gamma_n[2]$ has a large kernel, the \emph{palindromic Torelli group}, denoted $\ptor_n$. The second author found a generating set for $\ptor_n$ \cite{Ful15}, consisting of two types of generator, which we now describe. Let $Y$ be a palindromic free basis for $F_n$ (that is, the image of $X$ under some $\alpha \in \pia_n$). A \emph{doubled commutator transvection} maps some $c \in Y$ to $[a,b] \cdot c \cdot [a,b]^{\mathrm{rev}}$ for some $a, b \in Y \setminus \{c \}$, fixing the other members of $Y$. A \emph{separating $\pi$-twist} maps some $a, b, c \in Y$ to $s^{\mathrm{rev}} a s$, $s^{-1} b {(s^{\mathrm{rev}})}^{-1}$, and $s^{\mathrm{rev}} c s$, respectively (where $s = a^{-1}b^{-1}c^{-1}abc \in F_n$), and fixes the other members of $Y$. The following is proved.

\begin{theorem}[Fullarton, \cite{Ful15}] \label{paltor}For $n \geq 3$, the palindromic Torelli group $\ptor_n$ is generated by the set of doubled commutator transvections and separating $\pi$-twists. \end{theorem}

The definition\label{geomob} of a separating $\pi$-twist may seem clumsy; however, it has a geometric description that is far neater. Taking as a free basis the set of oriented loops seen in Figure~\ref{surface}, the Dehn twist about the separating curve $\chi$ acts upon the depicted surface's free fundamental group as a separating $\pi$-twist. To obtain any other separating $\pi$-twist, we must simply conjugate the one given by $\chi$ by an appropriate member of $\pia_n$. This geometric observation plays an important role in our proof of Theorem~\ref{hyptor}, as it indicates how to express a separating $\pi$-twist as a product of doubled commutator transvections, as we shall see in Section~\ref{secTor}.

Via the surjection onto $\Gamma_n[2]$, the following corollary of Theorem~\ref{paltor} yields a finite presentation for this principal congruence subgroup. The group $\Gamma_n[2]$ forms part of a short exact sequence involving the quotient of $\HOut$ by $\STn$, and it is this presentation that enables us to find generators for $\STn$. Notationally, we denote the images of $\sigma_i$, $\iota_i$ and $P_{ij}$ in $\GLn$ by $s_i$, $J_i$ and $Q_{ji}$, respectively. 

\begin{corollary}[Fullarton \cite{Ful15}, Kobayashi \cite{Kob15}, Margalit--Putman \cite{Mar13}]\label{LevelGen} For $n \geq 3$, the group $\Gamma_n[2]$ is generated by the set $\{ J_i, Q_{ji} \mid 1 \leq i \neq j \leq n \}$, subject to the defining relators:

\begin{multicols}{2}
\begin{enumerate}
    \item ${J_i}^2$
    \item $[J_i,J_j]$
    \item $(J_iQ_{ij})^2$
    \item $(J_jQ_{ij})^2$
    \item $[J_i, Q_{jk}]$
    \item $[Q_{ki}, Q_{kj}]$
    \item $[Q_{ji},Q_{ki}]$
    \item $[Q_{ij},Q_{kl}]$
    \item $[Q_{kj},Q_{ji}]{Q_{ki}}^{-2}$
    \item $(Q_{ij}{Q_{ik}}^{-1}Q_{ki}Q_{ji}Q_{jk}{Q_{kj}}^{-1})^2$,
\end{enumerate}
\end{multicols} where $1 \leq i,j,k,l \leq n$ are pairwise distinct.

 \end{corollary}

\subsection{Generating the group of hyperelliptic automorphisms}

Denote by $\mathrm{HAut}(F_n)$ the pre-image of $\mathrm{HOut}(F_n)$ inside $\Autfn$. This group $\mathrm{HAut}(F_n)$ of automorphisms is related in a natural way to the split short exact sequence
\[ 1 \longrightarrow F_n \longrightarrow F_n \rtimes \Z / 2 \longrightarrow \Z / 2 \longrightarrow 1 ,\] where the $\Z / 2$ action is by $\iota$. We explain and exploit this relationship in the proof of the following theorem, which is an $\Outfn$ version of a result of Jensen--McCammond--Meier \cite[Lemma 6.1]{JMM07}.

\begin{theorem}The group $\mathrm{HOut}(F_n)$ is generated by the images in $\Out (F_n)$ of $\pia_n$ and the automorphism $\tau$. \label{hgens}
\end{theorem}

\begin{proof}We abuse notation, and think of $\Z / 2$ as being generated by $\iota$. The group $F_n \rtimes \Z / 2$ is isomorphic to a free product of $n+1$ copies of $\Z / 2$, one generated by each element $x_i \iota$ ($1 \leq i \leq n$), and another by $\iota$. Due to a theorem of Krsti\'c \cite[Section 2]{Krs92}, the group of automorphisms of $F_n \rtimes \Z / 2$ that preserve the $F_n$ kernel is isomorphic to $\mathrm{HAut}(F_n) \leq \Autfn$, via restriction. 

As Jensen--McCammond--Meier observed \cite{JMM07}, the $F_n$ kernel is characteristic (it is the subgroup of even length words on the symbols $x_1 \iota , \dots, x_n \iota$), and hence \[ \Aut (\underbrace{(\Z / 2) \ast \dots \ast (\Z / 2)}_{n+1} ) \cong  \mathrm{HAut}(F_n) .\]

Automorphism groups of free products have known, finite presentations, when the factors in the free product are well-behaved (see Fouxe-Rabinovitch \cite{Fou40, Fou41} and Gilbert \cite{Gil87}). In this instance, we have that $ \Aut ((\Z / 2) \ast \dots \ast (\Z / 2))$ is generated by \emph{factor permutations} (that is, a permutation of the set $\mathcal{P} : = \{ \iota, x_1 \iota, \dots, x_n \iota \}$) and \emph{partial conjugations} (an automorphism conjugating one member of $\mathcal{P}$ by another, and fixing the others).

To prove the theorem, it suffices to demonstrate how these types of automorphism act upon the $F_n$ subgroup. As an example, consider the partial conjugation mapping $x_1 \iota$ to $(x_2 \iota) (x_1 \iota) (x_2 \iota)^{-1}$. This automorphism maps $x_1 = (x_1 \iota)\cdot \iota$ to \[ [(x_2 \iota) (x_1 \iota) (x_2 \iota)^{-1}] \cdot \iota = x_2 {x_1}^{-1} x_2 ,\] restricting to the product $\iota_1 P_{12}$ on $F_n$. As a non-palindromic example, consider the factor permutation induced by interchanging $x_1 \iota$ and $\iota$. For $i \neq 1$, this permutation maps $x_i = (x_i \iota) \cdot \iota$ to \[ (x_i \iota) \cdot ( x_1 \iota) = x_i {x_1}^{-1},\] and inverts $x_1$. This permutation thus restricts to the product $\iota_1 \iota \tau \iota$. It is left as a straightforward calculation to verify that the remaining partial conjugations and factor permutations can be written as a product of the generators given in the statement of the theorem.
\end{proof}

We close this section by remarking that Krsti\'c's theorem allows a finite presentation for $\HOut$ to easily be obtained from the finite presentation for the automorphism group of a free product given by Fouxe-Rabinovitch and Gilbert. However, for our purposes, the found generating set suffices.


\section{Generating the hyperelliptic Torelli group} \label{secTor}

In this section, we prove Theorem~\ref{hyptor}, showing that $\STn$ can be generated using doubled commutator transvections. Our strategy is to find a (finite) presentation for the image of $\HOut$ in $\GLn$ and lift its relators up to $\STn \lhd \HOut$ to form a normal generating set.

Let $\mathrm{HM}_n(\Z)$ denote the image of $\HOut$ in $\GLn$. A presentation for $\HMn$ is obtained using the following lemma.

\begin{lemma} \label{imageseq} The image of $\HMn$ in $\GLnt$ under mod 2 entry reduction is isomorphic to the symmetric group $\Omega_{n+1}$.
\end{lemma}

\begin{proof}Let $\Psi_2 : \HMn \to \GLnt$ denote the mod 2 reduction map. By Theorem~\ref{hgens}, the group $\HMn$ is generated by the principal level 2 congruence group $\Gamma_n[2]$, the permutation matrices in $\GLn$, and the matrix $T \in \HMn$ that has $\tau \in \HOut$ as a pre-image. It follows that the image of $\HMn$ under $\Psi_2$ is generated by the permutation matrices in $\GLnt$ and the matrix $\Psi_2(T) =: t$.

It is a straightforward exercise to show that the set $\mathcal{V}$ consisting of the standard basis row vectors $e_1, \dots , e_n$ of $(\Z / 2)^n$, and their sum $e_1 + \dots +e_n$, is closed under the action of $\Psi_2(\HMn)$ by right multiplication (that is, column operations). This action is clearly faithful, and induces a surjection of $\Psi_2(\HMn)$ onto the symmetric group on $\mathcal{V}$. Surjectivity follows due to the permutation matrices acting transitively on the standard basis vectors, with $t$ interchanging $e_1$ and $e_1 + \dots + e_n$. Note that this gives that $t$ joins onto the standard braid presentation of the group of permutation matrices as a `zeroth strand' to present $\Omega_{n+1} = \langle t, \bar s_1, \dots , \bar s_{n-1} \rangle$, where $\bar s_i$ is the 2-cycle $(e_i \hspace{0.05in} e_{i+1})$.
\end{proof}

\textbf{A comparison with braid groups.} Lemma~\ref{imageseq} provides the short exact sequence
\[ 1 \longrightarrow \Gamma_n[2] \longrightarrow \HMn \longrightarrow \Omega_{n+1} \longrightarrow 1. \] Pulling this back to $\HOut$, we obtain the sequence
\[ 1 \longrightarrow (\ppia_n)^{nc} \longrightarrow \HOut \longrightarrow \Omega_{n+1} \longrightarrow 1, \] where $(\ppia_n)^{nc}$ denotes the normal closure in $\HOut$ of the \emph{pure palindromic automorphism group} $\ppia_n$, which is the group generated by inversions and elementary palindromic automorphisms. We note that these sequences compare favorably with similar short exact sequences given by A'Campo \cite{Aca79} and Arnol'd \cite{Arn68}, with the braid group $B_{n+1}$ and pure braid group $PB_{n+1}$ playing the roles of $\HOut$ and $(\ppia_n)^{nc}$, respectively (see \cite[Section 2]{BMP15} for a discussion). 

One explanation for this relationship between $\HOut$ and $B_{n+1}$ is obtained by viewing the braid group $B_{2g+1}$ as the hyperelliptic mapping class group of an oriented surface $S_g^1$ with one boundary component. Since $S_g^1$ is a $K(F_{2g}, 1)$ space, the difference between $B_{2g}$ and $\mathrm{HOut}(F_{2g})$ is the choice between centralizing the hyperelliptic involution using self-maps of $S_g^1$ that fix the boundary pointwise or not. We note that the Dehn--Nielsen--Baer theorem also allows us to embed $B_{2g+1}$ as a subgroup of $\mathrm{HOut}(F_{2g})$.

Motivated by these geometric considerations, and the relationship between $(\ppia_n)^{nc}$ and the pure braid group $PB_{n+1}$, we pose the following problem.

\begin{problem}Describe the geometric and group theoretic properties of the group $(\ppia_n)^{nc}$.
\end{problem}

\textbf{A combinatorial construction.} As finite presentations are known for both the kernel and quotient of the above sequence, we may calculate a finite presentation for the group $\HMn$, which we do now.

\begin{theorem} \label{hmpres}Let $\epsilon \in \{ \pm 1 \}$. On the generating set
\[ \{  Q_{jk}, J_k, s_i, T \mid 1 \leq i, j \neq k \leq n \}, \] the group $\HMn$ is presented by the following list of defining relations:
\begin{enumerate} \item all relations in the finite presentation of $\Gamma_n[2] = \langle Q_{jk}, J_k \rangle$ given in Corollary~\ref{LevelGen},
\item the relations holding in the standard braid presentation of $\Omega_n = \langle s_i \rangle$.
\item $T^2 = Q_{12}Q_{13} \dots Q_{1n}$,
\item $T s_1 T = s_1 T s_1$,
\item all relations encoding the natural permutation action of $\Omega_n = \langle s_i \rangle$ on the finite generating set of $\Gamma_n[2] = \langle Q_{jk}, J_k \rangle$,
\item $J_1 \underline{T^{\epsilon} J_1 T^{-\epsilon}} = T^{ - 2 \epsilon}$,
\item $\underline{T^{\epsilon} J_k T^{-\epsilon}} J_k = {Q_{1k}}^{\epsilon}$ for $k > 1$,
\item $\underline{T Q_{12} T^{-1}} = Q_{12}$,
\item $J_1 J_2 Q_{21} {Q_{12}}^{-1} \underline{T Q_{21} T^{-1}} =  \left [ {Q_{13}}^{-1} \dots {Q_{1n}}^{-1} \right ] \cdot \left [ {Q_{23}}^{-1} \dots {Q_{2n}}^{-1} \right ]$,
\item $J_1 J_2 {Q_{12}}^{-1} Q_{21} \underline{T^{-1} Q_{21} T} =  \left [ {Q_{13}}^{-1} \dots {Q_{1n}}^{-1} \right ] \cdot \left [ {Q_{23}} \dots {Q_{2n}}\right ]$,
\item $\underline{T^{\epsilon} Q_{23} T^{-\epsilon}} = {Q_{13}}^{\epsilon} Q_{23}$.

\end{enumerate}
\end{theorem}

\begin{proof}Consider a short exact sequence of groups,
\[ 1 \longrightarrow K = \langle A \rangle \longrightarrow E = \langle B \rangle \longrightarrow G = \langle \bar B \mid R \rangle \longrightarrow 1. \] We will abuse notation, and identify $K$ with its image inside $E$. If the generating set $\bar B$ is the image of the set $B$ under the quotient by $K$, an elementary construction of combinatorial group theory produces a presentation for the group $E$ (see, for example, MKS). Its generators are those in $B$, with three different families of relations giving a defining list: \emph{kernel relations}, \emph{lifted relations} and \emph{conjugation relation}. We outline these in the next paragraph.

Any relation that holds in $K$ on the letters of $A^{\pm1}$, also holds in $E$. By rewriting the relation using letters of $B^{\pm1}$ instead, we obtain what we call a \emph{kernel relation} in $E$. Similarly, any relator $\bar r$ in $R$,  when rewritten using letters of $B^{\pm1}$ instead of $\bar B^{\pm1}$ in the obvious way, must lift to equal a member $\kappa$ of the kernel $K$. Writing this lift of $\bar r$ as $r$, we obtain a \emph{lifted relation} by setting $r = \kappa$, having rewritten $\kappa$ using letters of $B^{\pm1}$. Finally, we obtain a \emph{conjugation relation} by lifting some $\bar b \in \bar B$ to its preimage $b \in B$ and considering the conjugate $b^{\epsilon} a b^{-\epsilon}$ for some $a \in A$ and $\epsilon \in \{ \pm 1\}$. This conjugate must equal some $\kappa ' \in K$, so rewriting $\kappa '$ using $B$, we obtain the conjugation relation $b^{\epsilon} a b^{-\epsilon} = \kappa '$. 

Turning attention to the above short exact sequence for $\HMn$, we obtain the finite presentation in the statement of the theorem as follows. We generate $\HMn$ using the set given by Theorem~\ref{hgens}, and call this set $B$, as above. Notice that $B$ is partitioned into two subsets: one subset generating the kernel $\Gamma_n[2]$, and one mapping to the generating set for $\Omega_{n+1} \leq \GLnt$ given by Lemma~\ref{imageseq}.

Relations of type 1 in this theorem's statement correspond to kernel relations, in $\Gamma_n[2]$. Relations of type 2 -- 4 correspond to lifted relations, since, as was noted in the proof of Lemma~\ref{imageseq}, the quotient of $\HMn$ by $\Gamma_n[2]$ is presented by the standard braid presentation on the image of $B$.

Finally, relations of type 5 -- 11 correspond to conjugation relations. Type 5 deals with the lifts of the permutation matrices in $\Omega_{n+1} \leq \GLnt$, whereas types 6 -- 11 demonstrate the effect of conjugating $\Gamma_n[2]$ by $T$. Although these are not all written in the form `$b^{\epsilon} a b^{-\epsilon} = \kappa '$' as was given above, the portions of each relation corresponding to the conjugate $b^{\epsilon} a b^{-\epsilon}$ are underlined.  
\end{proof}

\textbf{Generating the hyperelliptic Torelli group.} With a presentation for $\HMn$ established, we now simply have to lift its relators to $\HOut$, to obtain a normal generating set for the hyperelliptic Torelli group, $\STn$. We do this now.

\begin{proof}[Proof of Theorem~\ref{hyptor}]Let $X$ denote the generating set for $\HOut$ from Theorem~\ref{hgens}, and let $\bar X$ denote its image in $\HMn$. We recall the short exact sequence
\[ 1 \longrightarrow \STn \longrightarrow \HOut = \langle X \rangle \longrightarrow \HMn = \langle \bar X \mid \mathcal{R} \rangle \longrightarrow 1, \] where $\mathcal{R}$ is a set of relators obtained from Theorem~\ref{hmpres}. It is a classical fact from combinatorial group theory that the kernel $\STn$ of such a sequence is generated by all conjugates in $\HOut$ of the obvious lifts of the relators in $\mathcal{R}$.

It is thus left to calculate appropriate lifts of the defining relators of $\HMn$ from Lemma~\ref{hmpres}. Relators of type 1 necessarily lift to products of doubled commutator transvections and separating $\pi$-twists, as they lift to members of the image in $\HOut$ of the palindromic Torelli group, $\ptor_n$ (recall Theorem~\ref{paltor}). Relators of types 2 -- 10 all lift trivially to $\HOut$, while relators of type 11 lift to doubled commutator transvections.

To finish proving the theorem, we use a geometric argument to show that any separating $\pi$-twist may be written as a product of doubled commutator transvections. Consider the generating set for the surface groups $\pi_1(S_{2g}^2) = F_{2g+1}$ seen in Figure~\ref{surface}. (The case for an even rank free group follows similarly, using a once-punctured surface). The Dehn twists about the separating curves $\chi$ and $\eta$ commute with the shown involution $s$, and are also conjugate by a mapping class that commutes with $s$. We see this by noting that $\chi$ and $\omega$ descend to simple closed curves surrounding three marked points on the punctured sphere $S_{2g}^2 / s$, and then applying the `change of coordinates principle' [primer] and the Birman--Hilden theorem.

Since $s$ maps to $\iota \in \mathrm{Out}(\pi_1(S_{2g+1}^2)) = \Out (F_{2g+1})$ via the Dehn--Nielsen--Baer homomorphism, the centralizer of the mapping class $[s]$ in $\Mod (S_{2g+1}^2)$ embeds into $\mathrm{HOut} (F_{2g+1})$. The twist about $\chi$ maps to a separating $\pi$-twist in $\mathrm{HOut} (F_{2g+1})$, as was noted in Section~\ref{geomob}. It is straightforward to check that the twist about $\eta$ maps to a product of doubled commutator transvections by $[a_1, b_1]$. These twists remain conjugate under the Dehn--Nielsen--Baer homomorphism, by some member of $\mathrm{HOut} (F_{2g+1})$. Since any separating $\pi$-twist is conjugate in $\mathrm{HOut} (F_{2g+1})$ to the image of the twist about $\chi$, we have thus shown that $\mathcal{ST}(2g+1)$ is generated by doubled commutator transvections.
\end{proof}

\begin{figure}[h!]
\centering
\begin{subfigure}[c]{\textwidth}
\centering
\begin{overpic}[width=5in]{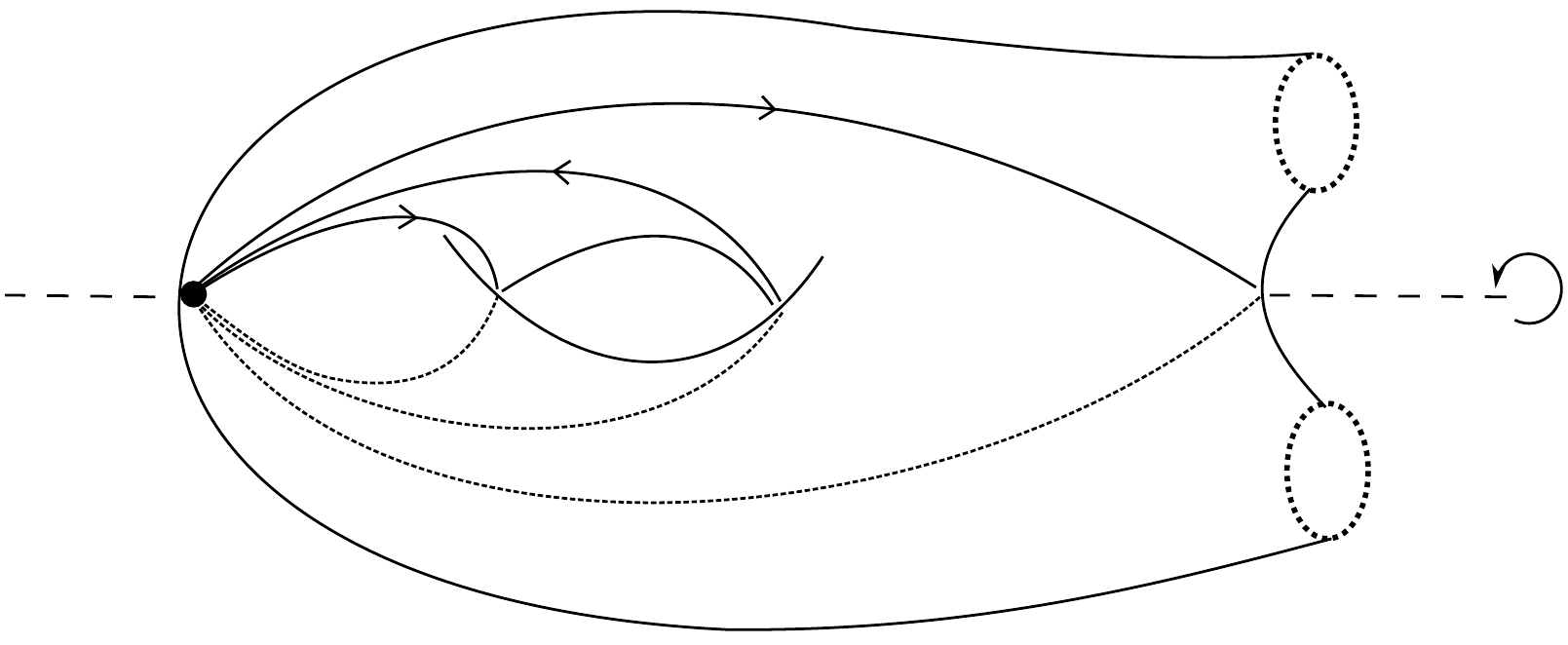}
\put(22,24.3){$x_1$}
\put(35,27){$x_2$}
\put(48,31.3){$x_3$}
\put(101,22){$s$}
\end{overpic}
\caption{Each loop $x_i$ depicted forms part of a free basis for the surface's free fundamental group that is acted upon by $\iota$ by the rotation $s$ by $\pi$. For higher genus surfaces, we extend the free basis using the obvious $s$-invariant loops.}
\label{Basis}
\end{subfigure}
\quad

\begin{subfigure}[c]{\textwidth}
\centering
\begin{overpic}[width=5in]{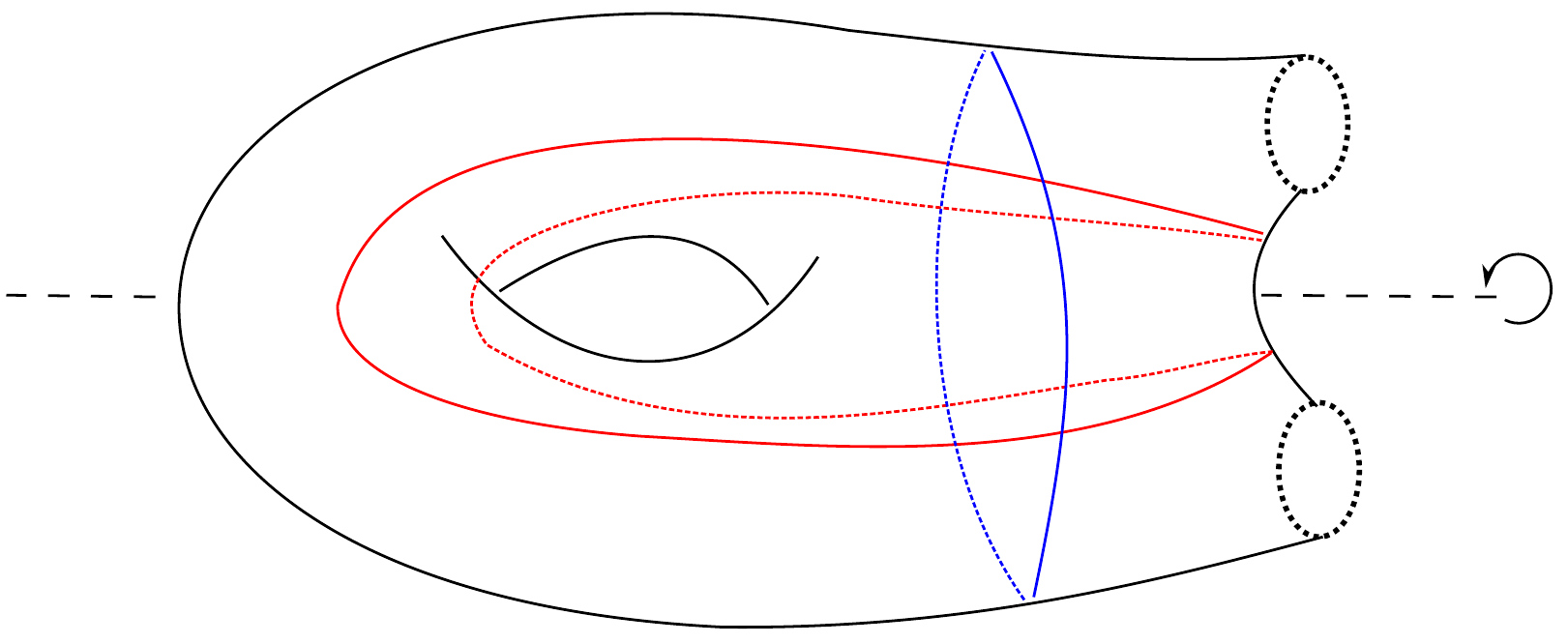}
\put(67,33){$\eta$}
\put(30,12){$\chi$}
\put(100.2,21.7){$s$}
\end{overpic}
\caption{The Dehn twists about $\chi$ and $\eta$ act on the basis from subfigure (a) as members of $\STn$.}
\label{DehnTwist}
\end{subfigure}
\caption{}
\label{surface}
\end{figure}

\section{The locus of hyperelliptic graphs}

In this section, we turn our attention to the subspace of Torelli space consisting of homology-marked graphs admitting a hyperelliptic involution. This space brings together the geometry of the hyperelliptic Torelli group $\STn$, for which it is an Eilenberg--Maclane space, and the period mapping, for which it behaves like a branch locus.

\subsection{Hyperelliptic marked graphs}\label{HyperMarked} Recall that by a \emph{hyperelliptic involution} of $F_n$, we mean a (perhaps outer) automorphism of $F_n$ with order 2 that acts as $-I$ on the abelianisation $H_1(F_n) \cong \Z^n$.

\textbf{A geometric model for $\STn$.} Denote by $\mathcal{L}_n$ the subspace of $CV_n$ consisting of marked graphs that are fixed by \emph{some} hyperelliptic involution in $\Outfn$ (that is, $\mathcal{L}_n$ is the union of the fixed sets of all such involutions in $\Outfn$). We say that such marked graphs \emph{admit a hyperelliptic involution}. Let $\mathcal{L}_n^{\iota}$ denote the fixed set of the explicit involution $\iota$. 

Any marked graph admits at most one hyperelliptic involution: if two were admitted, their difference would give a finite-order element of the Torelli group $\mathcal{I}_n$, which is a torsion-free group. Thus, by Lemma~\ref{hypuniq}, we conclude that $\mathcal{L}_n$ is a disjoint union of infinitely many copies of $\mathcal{L}_n^{\iota}$. The following theorem follows from work of Krsti\'c--Vogtmann. 

\begin{theorem}\label{cont}The space $\mathcal{L}_n^{\iota}$ is contractible. \end{theorem}
\begin{proof}Krsti\'c--Vogtmann established that the fixed set in the spine of outer space of any finite subgroup of $\Outfn$ is contractible. The spine's fixed set corresponding to $\iota$ is homotopy equivalent to $\mathcal{L}_n^{\iota}$, via an $\iota$-equivariant version of the usual deformation retraction of $CV_n$ onto its spine.
\end{proof}

The group $\HOut$ centralizes $\iota$ by definition, and so its action on $CV_n$ preserves $\loc$. This action has finite stabilizers, so $\STn$ being torsion-free gives the following corollary to Theorem~\ref{cont}.
\begin{corollary}\label{eilmac}The quotient $\hloci$ of $\loci$ by $\STn$ is an Eilenberg--Maclane space for $\STn$.
\end{corollary}

One motivation to study $\hloci$ (and $\hloc := \loc / \STn$, a disjoint union of copies of $\hloci$) is thus to better understand the topological properties of the group $\STn$.

\textbf{Branch-like behavior of the period mapping.} One key reason to analyze the period mapping is its similarity to the classical period mapping, defined on homology-marked Riemann surfaces. As discussed in Section 1, the classical map is a 2:1 cover onto its image, branched along the locus of hyperelliptic surfaces. As we established in Section~\ref{PeriodFiber}, the fibers of the period mapping for free groups are far from discrete in general, and so $\Phi$ seems rather unlike a branched covering map. However, restricting to a dense subspace of $\calT_n$, one can see that $\Phi$ behaves roughly like a 2:1 branched cover in the following sense.  

There is an open dense subset $U\subset \mathcal{Q}_n\cap \text{Im}(\Phi)$ consisting of forms $M$ with $\text{stab}_{\GL_n(\Z)}(M)$ trivial.  In $\calT_n$, the open dense subset  $V$ of trivalent graphs contains no graphs with separating vertices. Pulling back $U$ to $\calT_n$, define $W=\Phi^{-1}(U)\cap V$.  Then $W$ is open, dense and the full pre-image of a subset of $\mathcal{Q}_n$.  Moreover, by Theorem \ref{Fiber}, every fiber in $W$ is a union of contractible simplicial complexes and the stabilizer of each fiber is $\Z/2$ realized as the action of $\{\pm I\}$.  We expect that as $n$ gets large, `most' isometry types of graphs in $\calT_n$ are also 3-edge connected and 3-edge connected as in Example 2 of Section \ref{Examples}, in which case the fiber will consist of exactly 2 points interchanged by $\pm I$.


\textbf{Connectivity of $\hloci$.} In the light of Corollary~\ref{eilmac}, the connectivity of $\hloci$ is well-understood. The ultimate goal of this section is to show that the single surviving homotopy group of $\hloci$ (that is, $\pi_1 (\hloci) \cong \STn$) becomes trivial when certain degenerate marked graphs `at infinity' are added to $\hloci$. This compares favorably with a corresponding result of Brendle--Margalit--Putman regarding adding surfaces of compact type to the hyperelliptic locus of moduli space, and positively answers a question of Margalit, who asked if a free group analog of this result holds.

We make precise the notion of these graphs at infinity in Section~\ref{Degen}, and in Section~\ref{goal} prove that they trivialize the fundamental group of $\hloci$


\subsection{Marked graphs with degenerate rank} \label{Degen}
Outer space $CV_n$ is not quite a simplicial complex, but rather a simplicial complex with missing faces.  The simplicial completion $\overl{CV}_n$, \emph{i.e.} the simplicial complex obtained by adding in the missing faces, can be thought of in various ways.  Hatcher introduced the sphere complex and identified it with the simplicial completion of outer space \cite{Hat95}.  Later, Hatcher's sphere complex was interpreted in terms of free splittings--conjugacy classes of free product decompositions of $F_n$.  Free splittings form a poset under inclusion and the free splitting complex is just the simplicial realization of this poset. See \cite{ArSo11} for a careful identification of the 1-skeleton of the sphere complex and the free splitting complex.    

Via Bass-Serre theory, we obtain a third interpretation of a free splitting as a restricted class of actions of $F_n$ on simplicial trees.  A group action of $F_n$ on a simplicial tree $T$, denoted $F_n\curvearrowright T$, is \emph{minimal} if there are no $F_n$-invariant subtrees.  In particular, a minimal action has no global fixed points.  Define a \emph{free splitting} to be a minimal, faithful action of $F_n$ on a simplicial tree with trivial edge stabilizers.   Two free splittings are \emph{conjugate} if there is an $F_n$-equivariant isomorphism between their corresponding trees.  Much of the terminology and formalism we use to describe the free splitting complex in this section was adapted from work of Handel-Mosher \cite{HM13}.  For specific details on how to construct the free splitting complex from actions on trees, we refer the reader to \S 1 of their paper.

The geometric way to understand the simplicial completion $\overl{CV}_n$ can be phrased in terms of degenerate metrics on rank $n$ graphs. Given a rank $n$ graph $\Gamma$, consider a subgraph $\Delta=\coprod_{i=1}^m\Delta_i\subset \Gamma$.  We assume that $\Delta\neq \Gamma$.  Collapsing each component of $\Delta$ gives a new graph $\Gamma'$ together with special vertices $\delta_1,\ldots, \delta_m$, which are the images of $\Delta_1,\ldots, \Delta_m$ under the collapse map $c:\Gamma\rightarrow \Gamma'$.  

Let $p:\widetilde{\Gamma}\rightarrow \G$, and $p':\widetilde{\Gamma}'\rightarrow\G'$ denote the universal covers of $\G$ and $\G'$.  An action $F_n \curvearrowright \widetilde{\Gamma}$ induces a quotient action on $\widetilde{\Gamma'}$, where the stabilizer of a vertex $\tilde{v}\in \widetilde{\Gamma'}$ is a free group equal to the rank of $\pi_1(c^{-1}(p'(\tilde{v})))$. Thus the only vertices with possibly non-trivial vertex stabilizers are the lifts of the $\delta_i$.  $\Gamma'$ can therefore be identified with a degenerate metric on $\Gamma$, where the length of each edge in $\Delta$ is 0.  There is then a path in $\overl{CV}_n$ from $\Gamma$ to $\Gamma'$ obtained by scaling the metric on $\Delta$ by $t\in [0,1]$.  The metric on $\Gamma\setminus \Delta$ is also rescaled to maintain the volume 1 condition.  

Bass--Serre theory associates to $\Gamma'$ a free splitting of $F_n$ as $F_n=H_1*\cdots*H_m*F'$, where $H_i$ is the stabilizer associated to $\delta_i$, and $F'$ is the fundamental group of $\Gamma'$.  The collapse map $c:\G\rightarrow \G'$ induces a surjection $c_*:\pi_1(\G)\rightarrow \pi_1(\G')$ which is split because $\pi_1(\G')$ is free.  Obviously, there is not a unique splitting but the free splitting decomposition defined by $c$ gives a preferred splitting $s:\pi_1(\G')\rightarrow F'\leq\pi_1(\G)$ which is well-defined up to conjugation in $\pi_1(\G)$. Given a marking $\rho$ on $\Gamma$, define the \emph{marking of the degenerate graph} $\G'$ to be the conjugacy class $\rho'$ of the splitting $s:\pi_1(\G')\rightarrow \pi_1(\G)$. We refer to the pair $(\G',\rho')$ as a \emph{marked degenerate graph}. The special vertices $\delta_i$ are called \emph{degenerate vertices}, and the \emph{order} of a degenerate vertex $\delta_i$ is $\rk{H_i}$.

$\Out(F_n)$ acts on $CV_n$ by changing the marking, and this action extends to a simplicial action on $\overl{CV}_n$.  Note however that the extended action will no longer be properly discontinuous.  Suppose we have two rank $n$ marked graphs $(\G_1,\what{\rho}_1)$ and $(\G_2,\what{\rho}_2)$ and an element $\phi\in \Out(F_n)$ such that $\phi.(\G_1,\what{\rho}_1)=(\G_2,\what{\rho}_2)$.  If $c_1:(\G_1,\what{\rho}_1)\rightarrow(G_1',\what{\rho}_1')$ is a marked degeneration of $(\G_1,\what{\rho}_1)$, then there exists a marked degeneration $c_2:(\G_2,\what{\rho}_2)\rightarrow (\G_2',\what{\rho}_2')$ and a map $\psi$ such that the folllowing diagram
\begin{equation} \begin{aligned} \xymatrixrowsep{5pc}\xymatrixcolsep{5pc} \xymatrix{
(\G_1,\what{\rho}_1)\ar[d]^{c_1}\ar[r]^{\phi} &  (\G_2,\what{\rho}_2)\ar[d]^{c_2}\\
(\G_1',\what{\rho}_1')\ar[r]^{\phi} & (\G_2',\what{\rho}_2')
}
\end{aligned}\tag{$**$}
\end{equation}
commutes.  If we quotient $\overl{CV}_n$ by the Torelli subgroup $\calI_n$, the markings above become homology markings.  We then get a corresponding commutative diagram for the action of a matrix $A\in \GL_n(\Z)=\Out(F_n)/\calI_n$:
\begin{equation}\begin{aligned}\xymatrixrowsep{5pc}\xymatrixcolsep{5pc} \xymatrix{
(\G_1,\rho_1)\ar[d]^{c_1}\ar[r]^{A} &  (\G_2,\rho_2)\ar[d]^{c_2}\\
(\G_1',\rho'_1)\ar[r]^{A} & (\G_2',\rho'_2)}
\end{aligned}\tag{$*$}
\end{equation}

A maximal simplex in $\overl{CV}_n$ has dimension $3n-4$, with a barycenter represented by a 3-regular graph.  If this graph has no separating edges, then no single edge collapse lies in $\partial \overl{CV}_n=\overl{CV}_n\setminus CV_n$. Thus no codimension-1 simplices in reduced outer space correspond to degenerate marked graphs.  However, higher codimension faces of maximal simplices may correspond to degenerate metrics.  The next lemma will be used in the proof of Theorem C, when we study the local structure of the boundary $\partial \overl{CV_n}$.

\begin{lemma} \label{simplexChain}Let $M_1,M_2$ be maximal simplices in $CV_n$ whose closures $\overl{M_1},\overl{M_2}\subset \overl{\calT_n}$ share a degenerate face $F_0$. Suppose that the barycenter of $F_0$ corresponds to a marked rank $k\leq n$ rose $R_k$ with a single degenerate vertex of rank $n-k$ at the wedge point.  Then there exists a chain of maximal simplices $M_1=S_1, S_2,\ldots, S_m=M_2\in \calT_n$ such that $M_i\cap M_{i+1}\neq \emptyset $ for all  $i=1,\ldots,m-1$, and $F_0\subset \overl{S_i}\cap \overl{S_j}$ for all $i\neq j$. 
\end{lemma}
\begin{proof}
Let $(\G_1,\rho_1)\in M_1$ and $(\G_2,\rho_2)\in M_2$ denote the two homology-marked barycenters of $M_1$ and $M_2$ respectively, and $(R_k,\rho_0)$ the barycenter of $F_0$.  We can think of $R_k$ as obtained from $\G_i$, $i=1,2$, by collapsing a connected subgraph $\Delta_i$. Since each quotient $\G_i/\Delta_i$ is a wedge of circles, $\Delta_i$ contains a maximal tree $T_i$, and $R_k$ has a single degenerate vertex of rank $n-k$ at the wedge point $v_0\in R_k$.    

Let $c_i:(\G_i,\rho_i)\rightarrow (R_k,\rho_0)$, $i=1,2$, denote the two collapse maps. If we first collapse each maximal tree $T_i$,  subquotient $\G_i/T_i$ is a rank-$n$ rose $R_n$ which inherits a marking $(R_n, \sigma_i)\in M_i$. Then $c_i=g_i\circ f_i$ factors through collapse maps $f_i:(\G_i,\rho_i)\rightarrow (R_n, \sigma_i)$ and $g_i:(R_n, \sigma_i)\rightarrow (R_k,\rho_0)$. Therefore, in order to prove the lemma, it suffices to find a path from $(R_n, \sigma_1)$ to $(R_n, \sigma_2)$ consisting entirely of rank $n$ graphs which collapse onto $(R_k,\rho_0)$.  

After acting on the left by some element of $\GL_n(\Z)$, we may assume that $\sigma_1$ is the identity-marked rose.   Choose a basis $\{x_1,\ldots,x_n\}$ for $H_1(R_n)$ consisting represented by the wedge circles. We may choose $x_1,\ldots, x_n$ in such a way that the collapse map $g_1:(R_n, \sigma_1)\rightarrow (R_k,\rho_0)$ corresponds to the projection $\pi_k:H_1(R_n)\cong\Z^n\rightarrow H_1(R_{k})\cong\Z^{k}$ sending $\{x_1,\ldots,x_n\}\mapsto \{x_1,\ldots,x_k\}$.  The collapse map $g_2:(R_n, \sigma_1)\rightarrow (R_k,\rho_0)$ also induces a surjection $\pi:H_1(R_n)\cong\Z^n\rightarrow H_1(R_{k})\cong\Z^{k}$. From $(*)$ and the fact that both $(R_n, \sigma_1)$ and $(R_n, \sigma_2)$ degenerate to $(R_k,\rho_0)$ we obtain a commutative diagram:
\begin{equation}\begin{aligned}\xymatrixrowsep{5pc}\xymatrixcolsep{5pc} \xymatrix{
(R_n,\sigma_1)\ar[dr]_{g_1}\ar[rr]^{A} &&  (R_n,\sigma_2)\ar[dl]^{g_2}\\
 & (R_k, \rho_0)}
\end{aligned}
\end{equation}
for some $A\in \GL_n(\Z)$.  Passing to homology we obtain a commutative diagram
\begin{equation}\begin{aligned}\xymatrixrowsep{5pc}\xymatrixcolsep{5pc} \xymatrix{
\Z^n\ar[dr]_{\pi_k}\ar[rr]^{A} &&  \Z^n \ar[dl]^{\pi}\\
 & \Z^k}
\end{aligned}
\end{equation}
From diagram (2) we conclude that the matrix $A$ has block form 
\[\left(\begin{array}{c|c} I_{k} &0\\\hline
B & C
\end{array}\right),\]
where $I_k$ is the $k\times k$--identity matrix, $B$ is a $(n-k)\times k$--matrix, and $C\in \GL_{n-k}(\Z)$.  Clearly, $A$ can be written a product of elementary matrices $E_{ij}$ with $1\leq i\leq n$ and $n-k\leq j\leq n$. Each $E_{ij}$ can be realized as a change of marking on $R_n$ as follows.  
Let $e_i$ denote the oriented edge in $(R_n, \text{Id})$ corresponding to the generator $x_i$.  First expand an edge which drags one end of $e_i$ over the edge $e_j$ in the same orientation as $e_j$.  This divides the former edge $e_j$ into two edges $e_j'$ and $e_j''$ such that $l(e_j)=l(e_j')+l(e_j'')$, and creates a theta graph between the edge $e_i$, $e_j'$, and $e_j''$.  Now collapsing $e_j''$ returns a rose, and the marking on the edge $e_i$ has changed by $x_i\mapsto x_i+x_j$. See figure for a schematic.

Each stage of the marking change is either a rank-$n$ rose or a rank-$(n-1)$ rose wedged with a theta graph. By collapsing all of the edges corresponding to the generators $e_j,$ $e_j'$ or $e_j''$ for $n-k\leq j\leq n$, we obtain the marked rank-$k$ rose $(R_k,\rho_0)$. 
%
\end{proof}

\textbf{Closure of the hyperelliptic locus.} As noted in Section \ref{HyperMarked}, the fixed set $\mathcal{L}^\iota$ of the involution $\iota$ is a contractible subspace of $CV_n$.  One can consider the closure of $\mathcal{L}^\iota$ in the simplicial completion $\overl{CV}_n$, denoted $\overl{\mathcal{L}^\iota}$.  Since the action of $\iota$ extends continuously to $\overl{CV}_n$, it follows that all of $\overl{\mathcal{L}^\iota}$ is fixed by $\iota$. In terms of collapsing subgraphs, if $p\in \overl{CV}_n$ is a limit point of $\mathcal{L}^\iota$, then $p$ is represented by a marked degenerate graph which admits a hyperelliptic involution.  In other words, each component of the collapsed subgraph $\Delta$ must be invariant under $\iota$. The above discussion therefore also applies to $\overl{\mathcal{L}^\iota}$, with $\Out(F_n)$ replaced by its stabilizer $\HOut$, and we obtain corresponding commutative diagrams $(*)$ and $(**)$ which are also $\iota$-invariant. We denote the closure of a subset of $CV_n$ in $\overl{CV_n}$, or one of its quotients, by an overline.  Thus, the closures of $\mathcal{L}_n^\iota$, $\mathcal{L}_n$, or $\mathcal{HL}_n$ are denoted by $\overl{\mathcal{L}_n^\iota}$, $\overl{\mathcal{L}_n}$, or $\overl{\mathcal{HL}_n}$, respectively. 

Since any rose is hyperelliptic, the proof of Lemma \ref{simplexChain} goes through for $\overl{\mathcal{HL}_n^\iota}$ without any difficulties.  The change of marking matrix $A$, however, will not be an arbitrary matrix of $\GL_n(\Z)$, but rather of $\Gamma_n[2]$. The group $\Gamma_n[2]$ is generated by the $J_i$ and $Q_{ij}$ of Corollary \ref{LevelGen}. Each $J_i$ is realized as an isometry of the rose, hence, to find the path described in the proof, it suffices to show that each $Q_{ij}$ can be realized as a sequence of $\iota$-equivariant blow-ups and blow-downs.   

Such a sequence of blow-ups and blow-downs for the generator $Q_{ij}$ will be called a \emph{basic path} for $Q_{ij}$ and will be used extensively in the next section.  To describe a basic path, we will actually work in outer space.  Given generators $x_1,x_2\in F_n$, Figure~\ref{palpath} depicts a path in $CV_2$ that realizes the automorphism $P_{12}: x_1\mapsto x_2x_1x_2$. For $n >2$, we wedge a further $n-2$ circles at the base point in each graph. Clearly the image of $P_{12}$ in $\GL_n(\Z)$ is the generator $Q_{12}$. 

\begin{figure}[h]
\centering
\includegraphics[width=4in]{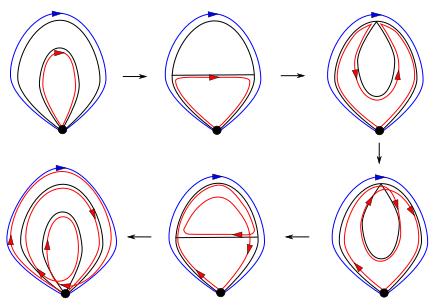}
\caption{In the initial marked rose, let the red loop indicate the generator $x_1$ and the blue loop $x_2$. Blowing up and down as indicated, the final rose is marked according to $P_{12}$, up to an isometry. Note that when transitioning from top to bottom row, we do not move in outer space: rather, we apply an isometry to the graph, before continuing on the path.}
\label{palpath}
\end{figure}

For completeness, we state the hyperelliptic version of Lemma \ref{simplexChain}.  
\begin{lemma} \label{simplexChain2}Let $M_1,M_2$ be maximal simplices in $\mathcal{HL}_n^\iota$ whose closures $\overl{M_1},\overl{M_2}\subset \overl{\mathcal{HL}_n^\iota}$ share a degenerate face $F_0$. Suppose that the barycenter of $F_0$ corresponds to a marked rank $k\leq n$ rose $R_k$ with a single degenerate vertex of rank $n-k$ at the wedge point.  Then there exists a chain of maximal simplices $M_1=S_1, S_2,\ldots, S_m=M_2\in \mathcal{L}_n^\iota$ such that $M_i\cap M_{i+1}\neq \emptyset $ for all  $i=1,\ldots, m-1$, and $F_0\subset \overl{S_i}\cap \overl{S_j}$ for all $i\neq j$. 
\end{lemma}

\subsection{Simple-connectivity of the hyperelliptic locus}\label{goal}
In this section we prove Theorem C. We remark that our techniques also prove the corresponding statement for the entire space $\calT_n$, using Lemma~\ref{simplexChain} and the fact that $\calI_n$ is normally generated by a single partial conjugation.

The strategy divides into two parts.  First, we show that in the closure of hyperelliptic Torelli space $\overl{\mathcal{HL}^\iota_n}$, loops corresponding to the generators for $\mathcal{ST}(n)$ all bound disks.  This proves that the inclusion $i:\mathcal{HL}^\iota_n\rightarrow \overl{\mathcal{HL}^\iota_n}$ induces the trivial map on fundamental groups.  In the second part, we show that there is no contribution to $\pi_1(\overl{\mathcal{HL}^\iota_n})$ coming from the boundary. To do this, we show that any loop $\gamma\subset \overl{\mathcal{HL}^\iota_n}$ based at the identity marked rose $(R_n,\text{Id})$ can be homotoped rel the basepoint to lie entirely within $\mathcal{HL}^\iota_n$.  By the first part we then conclude that $\gamma$ is can be contracted to a point.

Let $p_0=(R_n, \text{Id})$ denote the identity marked rank-$n$ rose, with all edge lengths equal. We will abuse notation and refer to this either as the point in $CV_n$ or $\calT_n$. We have
\begin{lemma} \label{inclusion}The homomorphism $i_*:\pi_1(\mathcal{HL}^\iota_n,p_0)\rightarrow \pi_1(\overl{\mathcal{HL}^\iota_n},p_0)$, induced by inclusion, is trivial.
\end{lemma}
\begin{proof}
From Theorem B, we know that $\mathcal{ST}(n)$ is generated by doubled commutator transvections, and that $\pi_1(\mathcal{HL}^\iota_n,p_0)\cong\mathcal{ST}(n)$.  To prove the lemma, it suffices to show that each loop representing a generating doubled commutator transvection based at $x_0$ bounds a disk in $\overl{\mathcal{HL}^\iota_n}$. We do this explicitly. 

Given generators $x_1,x_2,x_3$, consider the doubled commutator transvection \[\tau_{23}^1:x_1\mapsto[x_2,x_3]x_1[x_3^{-1},x_2^{-1}],\] acting as the identity on the other generators.  Starting from the rank-$n$ rose, $n\geq 3$, we can realize $\tau_{23}^1$ as the concatenation of basic paths for $Q_{12}\circ Q_{13}\circ Q_{12}^{-1}\circ Q_{13}^{-1}$.

At each step in the definition of the basic path, no assumptions were made about the lengths of the edges, only the homeomorphism type of each intermediate graph. If our basepoint $p_0$ corresponds to the metric graph with all edge lengths equal to $1/n$, we can carry out the basic path without changing the lengths of any of the $n-1$ edges labelled by generators $x_i$, $i=1$ or $3\leq i\leq n$, and such that at each stage, the sum of the lengths of edges involving $x_2$ is constantly $1/n$.  Let this path be denoted $\alpha_0$.  

Now consider the family of identity marked roses $p_t=(R_n^{t},\text{Id})$ in which the lengths of loops labelled $x_2,x_3$ are $(1-t)/n$, for $t\in[0,1]$ and the lengths of the remaining $n-2$ edges are each $\frac{n-2+2t}{n(n-2)}$. The collection of basic paths $\alpha_t$, where $\alpha_t$ starts at $p_t$, defines a homotopy $H:S^1\times[0,1]\rightarrow \overl{\mathcal{HL}^\iota_n}$ from the basic path $\alpha_0$ at $p_0$ to a basic path $\alpha_1$ based at $R_n^1$, a rank-$(n-2)$ degenerate rose in the boundary (see Figure~\ref{degendisk}).  However, since the lengths of the loops $x_2$ and $x_3$ in $R_{n}^1$ are 0, it follows that $\alpha_1$ is the constant path.  Thus $H$ factors through a map of the disk $h:D^2\rightarrow \overl{\mathcal{HL}^\iota_n}$, so $\alpha_1$ is nullhomotopic. Since all doubled commutator transvections are conjugate to $\tau_{23}^1$, they are also represented by nullhomotopic loops in $\overl{\mathcal{HL}^\iota_n}$, which proves the lemma.  
\begin{figure}[h]
\centering
\begin{overpic}[width=4in]{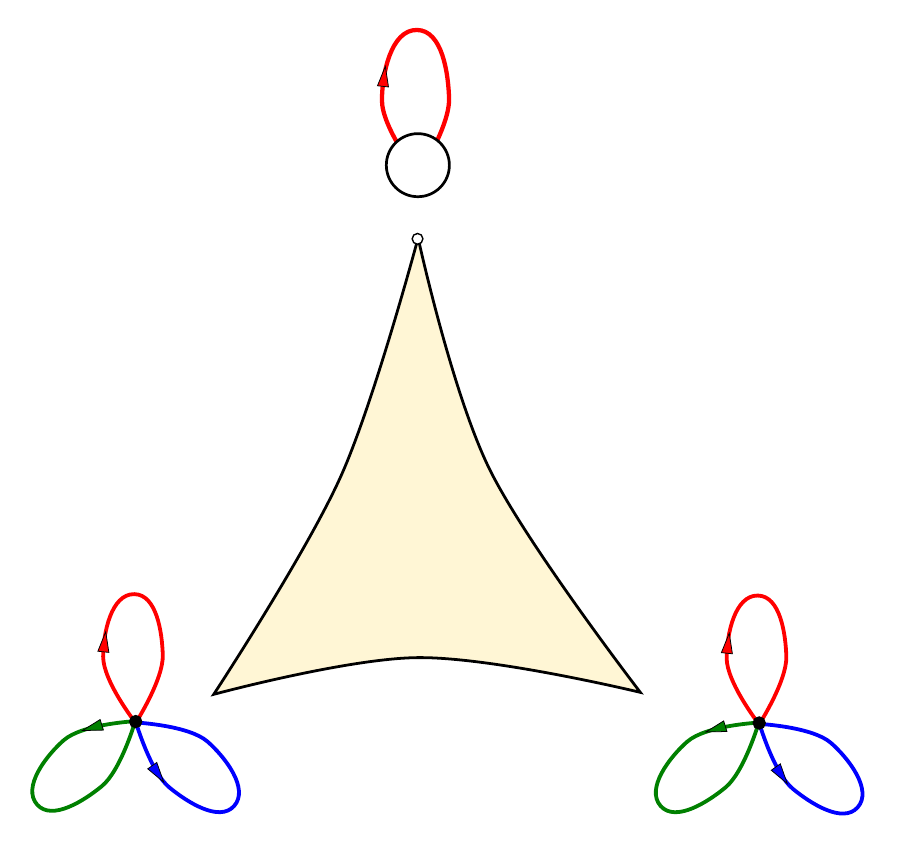}
\put(73,16){$\tau_{23}^1 \cdot$}
\put(44.5,76){$F_2$}
\end{overpic}
\caption{The disk in $\overl{\mathcal{HL}_n^{\iota}}$ over which $\tau_{23}^1$ is nullhomotopic.}
\label{degendisk}
\end{figure}
\end{proof}

Lemma \ref{inclusion} accomplishes the first part of our strategy for proving Theorem \ref{locusconn}.  The second part of our strategy is accomplished by the following lemma.
\begin{lemma} Let $\gamma\subset \overl{\mathcal{HL}^\iota_n}$ be any loop based at $p_0$.  Then $\gamma$ is homotopic rel $p_0$ to $\gamma''$, where $\gamma''$ lies in $\mathcal{HL}^\iota_n$.
\end{lemma}
\begin{proof} 

Let $\gamma$ be any loop in $\overl{\mathcal{HL}^\iota_n}$, based at $p_0$. If $\gamma$ lies entirely within $\mathcal{HL}^\iota_n$, then by Lemma \ref{inclusion}, it is nullhomotopic, and we are done.  We can therefore assume that $\gamma$ meets $\partial \overl{\mathcal{HL}^\iota_n}$. Observe that for $n\geq 3$, the space $\partial \overl{\mathcal{HL}^\iota_n}$ contains the 1-skeleton of the complex $\overl{\mathcal{HL}^\iota_n}$. Thus, up to free homotopy, we may represent $\gamma$ as a path in this 1-skeleton that traverses finitely many edges (and hence passes through finitely many vertices). Each arc of $\gamma$ traveling along a single edge may be pushed, relative to its endpoints, into the interior of a maximal simplex of $\partial \overl{\mathcal{HL}^\iota_n}$ (that is, passes through $\mathcal{HL}^\iota_n$). 

The above procedure allows us to assume that $\gamma$ is of the form $\gamma=\gamma_1'*\cdots* \gamma_r'$, where $*$ denotes concatenation of paths, and each $\gamma_i':[0,1] \rightarrow \overl{\mathcal{HL}^\iota_n}$ is a path satisfying: $\gamma_i'(0,1)\subset \text{Int}(\Delta_i)$, for some maximal simplex $\Delta_i$. Hence, each $\gamma_i'$ meets $\partial\overl{\mathcal{HL}^\iota_n}$ at worst at its endpoints.

If $\gamma_i'(1)\in \partial \overl{\mathcal{HL}^\iota_n}$, we will call it a \emph{degenerate point} of $\gamma$.  To prove the lemma, it suffices to show that each degenerate point of $\gamma$ can be pushed into the interior of $\overl{\mathcal{HL}^\iota_n}$. Let $z_i=\gamma_i'(1)$ for some $i$, $1\leq i\leq r$, and suppose that $z_i$ is a degenerate point corresponding to a degenerate, homology-marked hyperelliptic graph $(\Gamma_0,\rho_0)$. Note that $(\G_0, \rho_0)$ is a marked $S^1$, with a single degenerate vertex of rank $n-1$.

We now apply Lemma~\ref{simplexChain2} to the degenerate point $z_i$.  Let $\alpha_i$ be the path from $\Delta_i$ to $\Delta_{i+1}$ given by Lemma~\ref{simplexChain2}. The path $\alpha_i$ does not depend on the lengths of the edges of intermediate graphs. This means that if $\Delta_i=S_1, S_2,\ldots, S_m=\Delta_{i+1}$ is the chain of maximal simplices containing $\alpha_i$, then $\overl{S}_1\cup\cdots\cup \overl{S}_m$ actually contains the cone on $\alpha_i$ with cone point $z_i$. Using this cone, we can clearly homotope $\gamma$ so that $\gamma'_i(1)\in \mathcal{HL}^\iota_n$.  The proof is then completed by induction on the number of degenerate points of $\gamma'$.  
\end{proof}

\bibliographystyle{abbrv}
\bibliography{HyperBib}

\end{document}